\documentclass[11pt,reqno]{amsart}
\usepackage{graphicx}
\usepackage{color}
\usepackage{amsmath,amssymb,amsthm,amsfonts}
\usepackage{graphicx}
\usepackage{color}
\usepackage{multicol}

\usepackage[numbers,sort&compress]{natbib}
\def\R{\mathbb{R}}
\def\a{\alpha}
\def\b{\beta}
\def\k{\chi}
\def\ii{\int_0^\infty}
\def\be{\begin{equation}}
\def\ee{\end{equation}}

\def\C{\mathcal{C}}

\newtheorem{lemma}{\bf Lemma}[section]
\newtheorem{theorem}{\bf Theorem}[section]
\newtheorem{proposition}{\bf Proposition}[section]

\numberwithin{equation}{section}
\newcommand{\abs}[1]{\left\vert#1\right\vert}
\newcommand{\norm}[1]{\left\lVert#1\right\rVert}

\begin{document}
\author{Jingyu Li}
\address{School of Mathematics and Statistics, Northeast Normal University, Changchun, 130024, P R China}
\email{lijy645@nenu.edu.cn}

\author{Zhian Wang}
\address{Department of Applied Mathematics, Hong Kong Polytechnic University, Hung Hom, Kowloon, Hong Kong}
\email{mawza@polyu.edu.hk}
\title[Convergence to traveling waves of a PDE-ODE chemotaxis system]{Convergence to traveling waves of a singular PDE-ODE hybrid chemotaxis system in the half space}

\begin{abstract}
This paper is concerned with the asymptotic stability of the initial-boundary value problem of a singular PDE-ODE hybrid chemotaxis system in the half space $\R_+=[0, \infty)$. We show that when the non-zero flux boundary condition at $x=0$ is prescribed and the initial data are suitably chosen, the solution of the initial-boundary value problem converges, as time tend to infinity, to a shifted traveling wavefront restricted in the half space  $[0,\infty)$ where the wave profile and speed are uniquely selected by the boundary flux data. The results are proved by a Cole-Hopf type transformation and weighted energy estimates along with the technique of taking {\color{black} the} anti-derivative.
\end{abstract}

\subjclass[2000]{35A01, 35B40, 35K57, 35Q92, 92C17}

\keywords{Singular chemotaxis, shifted traveling waves, half space, boundary layer effect, convergence}
\maketitle

\section{Introduction}

This paper is concerned with the following PDE-ODE hybrid chemotaxis model
\begin{equation}\label{omn}
\begin{cases}
u_t=[Du_x-\xi u (\ln c)_x]_x, \\
c_t=-\mu uc,
\end{cases}
\end{equation}
which was proposed in \cite{LSN} to describe interaction between signaling molecules vascular endothelial growth factor (VEGF) and vascular endothelial cells during the initiation of tumor angiogenesis (see also in \cite{CPZ1, CPZ3}), where $u(x,t)$ and $c(x,t)$ denote the density of vascular endothelial cells and concentration of VEGF, respectively. The parameter $D$ denotes the cell diffusion rate, $\xi>0$ is referred to as the chemotactic coefficient measuring the strength of chemotaxis and $\mu>0$ denotes the degradation rate of the chemical VEGF. Note that the chemical diffusion is neglected since it is far less important than its interaction with endothelial cells (cf. \cite{LSN}). Among other things, the system \eqref{omn} can be regarded as a particular form of the well-known Keller-Segel system (cf. \cite{KS}) describing the propagation of traveling wave band formed by bacterial chemotaxis observed in the experiment of Adler \cite{Adler} where $u$ denotes the bacterial density and $c$ the concentration of nutrients. A distinguishing feature of the chemotaxis system \eqref{omn} lies in the singular logarithmic sensitivity based on the Weber-Fechner law (cf. \cite{KS}). The mathematical derivation such kind of chemotaxis models has been previously given in \cite{Othmer97, Levine97}.

As a major phenomenon observed in the various experiments for chemotaxis (cf. \cite{Adler, Gold, Welch}), the traveling wave has become one of the most genuinely interesting research topics in chemotaxis and a large amount of pivotal  theoretical works have been developed (cf. \cite{Horst2, Horst-Stev, wang12, Salako-Shen1, Salako-Shen2, Nadin, Ou}). By convention, traveling wave solution of $\eqref{omn}$ is a non-constant smooth solution over $\R=(-\infty,+\infty)$ in a self-similar ansatz in the form of
\begin{equation*}
(u,c)(x,t)=(U,\C)(z),\ z=x-st, \ x\in (-\infty, \infty)
\end{equation*}
with (general) boundary conditions (i.e. asymptotic states)
\begin{equation*}\label{twbc}
U(\pm\infty)=u_\pm,\ \C(\pm\infty)=c_\pm,\  U'(\pm\infty)=\C'(\pm\infty)=0,
\end{equation*}
where $'=\frac{d}{dz}$, $z$ is called the moving coordinate and $s>0$ is the wave speed. In addition the asymptotic states $u_\pm\geq0$,  $c_\pm\geq0$ due to their biological nature.
Upon a substitution of {\color{black} the} above wave ansatz into (\ref{omn}), $(U,\C)$ satisfies
\begin{equation}\label{tw}
\begin{cases}
-sU'+\xi \left(U(\ln\C)'\right)'=DU'',\\
s\C'=\mu U\C.
\end{cases}
\end{equation}
In view of the second equation of \eqref{tw} and the fact $\mu>0$, we have $\C'\geq0$, which implies $c_+>0$. This, along with $\C'(+\infty)=0$, yields $u_+=0$. With these observations and the special structure of ODE system \eqref{tw}, one can solve the solution $(U, \C)$ explicitly as (see \cite{wang12})
\begin{equation}\label{expression}
U(z)=\frac{\frac{s^2}{\xi\mu}}{\kappa e^{\frac{s}{D}z}+1}, \ \ \C(z)=c_+\left(\frac{\kappa e^{\frac{s}{D}z}}{\kappa e^{\frac{s}{D}z}+1}\right)^{\frac{D}{\xi}}
\end{equation}
with a unique wave speed $s=(\xi \mu u_-)^{1/2}$ for given $u_->0$. {\color{black} Note that the solution given in \eqref{expression} is unique up to a translation, where $\kappa>0$ corresponds to a translation constant. Therefore hereafter we shall assume $\kappa=1$ without loss of generality (otherwise we consider a shifted solution $(U(z+\tau), \mathcal{C}(z+\tau))$ with $\tau=-\frac{D}{s}\ln \kappa$}. Then the following results can be immediately verified:
\begin{equation}\label{property}
U'<0,\ \C'>0 \text{ for }z\in(-\infty,+\infty), \ u_+=0,\ c_+>0, \ u_->0,\ c_-=0.
\end{equation}
From the afore-mentioned results, we find that there are two free (arbitrary) parameters $u_-$ and $c_+$, which in turn indicates that the system \eqref{omn} may have infinite many traveling wave profiles over $(-\infty, \infty)$. This is, however,  unrealistic for a practical problem. Hence one immediate question concerned is how the free  parameters $u_-$ and $c_+$ can be determined ? Since the above traveling wave solution in the whole space $\R=(-\infty, \infty)$ propagates from the left to the right, it has been assumed that the source/force driving the traveling waves comes from the place at $-\infty$, which is not physical since the site where the source/force is placed could not be infinitely far for a realistic problem. Most of (if not all) real experiments are indeed completed in various confined apparatuses.  For instance, the experiment finding the chemotactic traveling waves of bacterial in the celebrated work \cite{Adler} was performed in a capillary tube, and the one observing the rippling wave patterns of myxobacteria was done in a gasket apparatus \cite{Welch}.  This motivates us to consider the problem in the half space with a physical boundary and see if the free parameters appearing in the traveling wave profiles  can be determined by the prescribed boundary conditions, and furthermore investigate whether the solution of the resulting initial-boundary value problem converges to any traveling wave profile (i.e. stability of traveling waves).   Toward this end,  in this paper, we consider the system \eqref{omn} for $(x,t) \in \R_+ \times \R_+$ where $\R_+=[0, \infty)$, with the following initial and boundary conditions
\begin{equation}\label{omn-init}
(u,c)(x,0)=(u_0,c_0)(x) \to (0, c_+) \ \text{as} \ x \to \infty,
\end{equation}
\begin{equation}\label{omn-bound}
(Du_{x}-\xi u(\ln c)_x)(0,t)=\eta(t), \ t\in\R_+
\end{equation}
where $c_+>0$ and $\eta(t)$ is a function of time $t$ convergent at $\infty$:
\begin{equation}\label{eta}
\eta(t)\rightarrow\eta_- \ne 0 \ \text{ as }t\rightarrow\infty.
\end{equation}
Note that no boundary condition is imposed for the solution component $c$ since its equation is just an ODE without a spatial structure. The boundary condition \eqref{omn-bound} means that cell density has a non-zero flux across the boundary. This is indeed the case in view of the process that the model \eqref{omn} describes: migration of vascular endothelial cells (denoted by $u$) towards the signaling molecules VEGF (cf. \cite{LSN}) (denoted by $c$) where the vascular endothelial cells come from the blood in the vessel by crossing the vessel wall. Hence if we regard the vessel wall as the physical boundary of our problem, then the non-zero flux boundary as \eqref{omn-bound} is a natural choice.  It is the purpose of this paper to exploit the asymptotic behavior of solutions of \eqref{omn} with \eqref{omn-init}-\eqref{omn-bound}. Specifically we shall prove that the solution of \eqref{omn} with \eqref{omn-init}-\eqref{omn-bound} on $[0,\infty)$  will approach the traveling wave solution of \eqref{omn} in $(-\infty, \infty)$ restricted on $[0,\infty)$ where the asymptotic state $u_-$ and wave speed will be uniquely determined by the asymptotic boundary datum $\eta_-$ and the parameter $c_+$ is nothing but the asymptotic state of initial value $c_0$ as $x \to \infty$. In other words, though the system \eqref{omn} has infinite many traveling waves profiles over $(-\infty, \infty)$, the solution of \eqref{omn} over half space with initial-boundary conditions \eqref{omn-init}-\eqref{omn-bound} will converge to a uniquely selected traveling wavefront profile.
Our results will not only address the determination of free parameters, but also assert that the movement of vascular endothelial cells in the initiation of tumor angiogenesis (cf. \cite{LSN}) can stabilize into a traveling wavefront profile with a unique wave speed.

To exploit the stabilization problem depicted above, one has to overcome some obstacles appearing in the analysis of the model (\ref{omn}). One is the singularity at ${c=0}$ in the first equation of \eqref{omn}. This singularity has its biological and mathematical grounds (cf. \cite{Othmer97, Kalinin}) and is irreplaceable to generate traveling wave patterns (cf. \cite{Lui-Wang}) although it brings great challenges to the stability analysis of traveling wave solutions. The other is the second equation of (\ref{omn}) is an ODE lacking a spatial structure and as a result the regularity of solution component $c$ may be problematic.   Hence working on the system \eqref{omn} directly will be rather difficult. In this paper, as usual, we break these barriers by employing a Cole-Hopf type transformation (cf. \cite{Othmer97, Wang08})
\begin{equation}\label{transformation}
v\triangleq-\frac{(\ln c)_x}{\mu}=-\frac{1}{\mu} \frac{c_x}{c},
\end{equation}
which turns the system \eqref{omn} into a parabolic-hyperbolic system without singularity
\begin{equation}\label{ph}
\begin{cases}
u_t-\chi(uv)_x=Du_{xx},\\
v_t-u_x=0,
\end{cases}
\end{equation}
with $\chi=\mu \xi>0$.
The initial-boundary conditions \eqref{omn-init}-\eqref{omn-bound} become
\begin{equation}\label{initial date}
(u,v)(x,0)=(u_0,v_0)(x) \ \to (0,v_+)  \ \ \text{as} \ x \to \infty,
\end{equation}
\begin{equation}\label{new-boundary data}
(Du_{x}+\chi uv)(0,t)=\eta(t), \ t\in\R_+
\end{equation}
where $v_+=-\displaystyle \frac{1}{\mu} \lim\limits_{x \to +\infty} \frac{c_{0x}}{c_0}=0$.
Therefore our plan is to work on the transformed problem \eqref{ph}-\eqref{new-boundary data} first and then transfer the results back to the original problem
\eqref{omn} with \eqref{omn-init}-\eqref{omn-bound} by solving $c$ in terms of $v$ from \eqref{transformation}. The detailed results and some new ideas developed in this paper will be elaborated in section 2 when appropriate.

Next we recall some existing results related to the transformed system \eqref{ph} and hence the original chemotaxis system \eqref{omn}. First, the one-dimensional problem has been studied extensively from various aspects such as the existence/stability of traveling wave solutions \cite{jin13, Li09, Li10, Lij13,Davis-Marangell, Mei-peng-wang}, global dynamics of solutions in the whole space \cite{GXZZ, Li-pan-zhao, MWZ-Indiana-2018, zhang-tan-sun} or in the bounded interval \cite{Hou2, Li-Zhao-JDE, TWW, Li112, zhang07}. Recently the half-space problem of \eqref{ph} on $[0,\infty)$ with non-homogeneous Robin boundary conditions on $u$ was considered in \cite{Deng} where the point-wise estimates of solutions converging to constant steady states was derived. {\color{black}The multidimensional form corresponding to \eqref{ph} reads (cf. \cite{Wang-xiang-yu})
\begin{equation}\label{phn}
\begin{cases}
u_t-\nabla \cdot (\chi u \vec{v})=D\Delta u,\\
\vec{v}_t-\nabla u=0,
\end{cases}
\end{equation}
where $\vec{v}:=-\frac{1}{\mu} \frac{\nabla c}{c}$ is a vector.}
In the whole space $\R^d$ ($d\geq 2$), when the initial datum is close to the constant ground state $(\bar{u}, {\bf 0})$, numerous results have been  obtained to the system (\ref{ph}).  First the local well-posedness and blowup criteria of large-amplitude classical solutions have been established in \cite{Fan-zhao, Li111} where in particular the global well-posedness and large time behavior of classical solutions have been obtained in \cite{Li111} if $(u_0-\bar{u} , {\bf v_0})\in H^s(\R^d)$ for $s>\frac d2+1$ and $\norm{(u_0-\bar{u} , {\bf v_0})}_{H^s\times H^s}$ is small. Later, Hao \cite{Hao} established the global existence of mild solutions in the critical Besov space $\dot{B}_{2,1}^{-\frac12}\times (\dot{B}_{2,1}^{-\frac12})^d $ with minimal regularity in the Chemin-Lerner space framework. The global well-posedness of strong solutions of \eqref{ph} in $\R^3$ was established in \cite{DL} if $\norm{(u_0-\bar{u} , {\bf v_0})}_{L^2\times H^1}$ is small. If the initial datum has a higher regularity such that $\norm{(u_0-\bar{u} , {\bf v_0})}_{H^2\times H^1}$ is small, the algebraic decay of solutions was further derived in \cite{DL}. Recently, Wang, Xiang and Yu \cite{Wang-xiang-yu} established the global existence and time decay rates of solutions of \eqref{ph} in $\R^d$ for $d=2,3$ if $(u_0-\bar{u} , {\bf v_0})\in H^2(\R^d)$ and $\norm{(u_0-\bar{u}, {\bf v_0})}_{H^1\times H^1}$ is small.    In the multidimensional bounded domain $\Omega \subset \R^d(d=2,3)$, the global existence and decay properties of solutions under Neumann boundary conditions were obtained in \cite{Li112} for small data.
When the chemical diffusion is considered, namely the second equation of \eqref{omn} is replaced by $c_t= {\color{black}c_{xx}}-\mu uc$, we refer to \cite{Li-Wang-MBS, Li11, Li14, PWZ, TWW, Wang-xiang-yu, Win1, Win2} and references therein for various interesting results.

Finally we state the organization of the rest of this paper. In section 2, we shall derive some preliminary results and then state our main results on both the transformed system \eqref{ph} and the original system \eqref{omn}. In section 3, we prove our main results.

\section{Preliminaries and main results}

Before proceeding, we clarify some notations used throughout this paper for convenience.\\

{\bf Notations}. In the sequel, we use $C>0$ to denote a generic constant which may change from line to line. $H^m(\mathbb{R}_+)(m\geq0)$ is the usual Sobolev space whose norm is abbreviated as
$\|f\|_m:=\sum\limits_{k=0}^{m}\|\partial_x^kf\|$ with $\|f\|:=\|f\|_{L^2(\mathbb{R}_+)}$, and $H^m_w(\mathbb{R}_+)$ denotes the weighted Sobolev space of measurable function $f$ such that $\sqrt{w}\partial_x^jf\in L^2(\mathbb{R}_+)$ for $0\leq j\leq m$ with norm
$\|f\|_{m,w}:=\sum\limits_{k=0}^{m}\|\sqrt{w}\partial_x^kf\|$ and $\|f\|_w:=\|\sqrt{w}f\|_{L^2(\mathbb{R}_+)}$.\\

\subsection{Wave selection}
As mentioned before, we shall work on the transformed problem \eqref{ph}-\eqref{new-boundary data} first. For this, we need to study the traveling wave solutions of \eqref{ph} over $(-\infty, \infty)$ and identify which wave will be selected by the given initial boundary value problem \eqref{ph}-\eqref{new-boundary data}. To this end,  we substitute the wave ansatz
$$(U,V)(z)=(u,v)(x,t), z=x-st \in (-\infty, \infty)$$
into \eqref{ph} and obtain that
\begin{equation}\label{traveling wave equation}
\begin{cases}
-sU'-\chi(UV)'=DU'',\\
-sV'-U'=0
\end{cases}
\end{equation}
From \eqref{transformation} and \eqref{property}, one can easily see that
\[V(z)=\frac{-\C'(z)}{\mu\C(z)}<0,\ V(+\infty)=-\lim_{z\rightarrow+\infty}\frac{\C'(z)}{\mu\C(z)}=0.\]
Hence, by \eqref{property}, the boundary conditions of \eqref{traveling wave equation} read
\begin{equation}\label{boundary condition}
U(+\infty)=V(+\infty)=0,\ U(-\infty)=u_-,~V(-\infty)=v_-<0,~U'(\pm\infty)=V'(\pm\infty)=0.
\end{equation}
Then integrating (\ref{traveling wave equation}) over $[z,+\infty)$ yields
\begin{equation}\label{3-4}
\begin{cases}
-sU-\chi UV=DU',\\
-sV=U.
\end{cases}
\end{equation}
The second equation of \eqref{3-4} implies $-sv_-=u_-$, which in combination with \eqref{expression} leads to
\begin{equation}\label{1-6}
s=\sqrt{\chi u_-}, \ \  v_-=-\sqrt{u_-/\chi}<0.
\end{equation}
Now by \eqref{property} and \eqref{1-6}, we obtain the existence of traveling wave solutions to the transformed system $\eqref{ph}$.
\begin{lemma}\label{etw}
Assume that $u_-$ and $v_-$ satisfy $\eqref{1-6}$. Then the system \eqref{traveling wave equation}-\eqref{boundary condition} has a monotone traveling wave solution $(U,V)(z)=(U,V)(x-st)$ which is unique up to a translation and has an explicit form:
\begin{equation}\label{explicit}
U(z)=\frac{u_-}{e^{\frac{s}{D}z}+1},\ V(z)=\frac{v_-}{e^{\frac{s}{D}z}+1}<0
\end{equation}
where the wave speed $s$ is given by $\eqref{1-6}$ and $U'<0,V'>0$.
\end{lemma}

In the following, we will first study the convergence of solutions of \eqref{ph}-\eqref{new-boundary data} to a shifted traveling wave solution $(U,V)(z)$ restricted on $[0, \infty)$, and then transfer the results back to the original chemotaxis system \eqref{omn} with \eqref{omn-init}-\eqref{omn-bound}. We note that the system \eqref{ph}-\eqref{new-boundary data} is confined on the half space $[0,+\infty)$ with a non-zero flux boundary condition given at $x=0$, while the traveling wave solution $(U,V)(z)$ is defined on the whole space $(-\infty,+\infty)$ with boundary condition $(U,V)(-\infty)=(u_-,v_-)$. 
It is easy to see that the traveling wave $(U,V)(x-st)$ at $x=0$ satisfies
\[(DU_x+\chi UV)(x-st)|_{x=0}=(DU'+\chi UV)(-st)\rightarrow\chi u_-v_- \ \text{ as }t\rightarrow+\infty.\]
Since we expect that the solution $(u,v)$ of \eqref{ph}-\eqref{new-boundary data} converges to the traveling wave $(U,V)$ asymptotically in time, owing to \eqref{eta} and \eqref{new-boundary data}, it is necessary that
\begin{equation}\label{2.13}
\eta_-=\chi u_-v_-<0
\end{equation}
where the fact $v_-<0$ has been used. The condition $\eta_-<0$ means that there exists a continuous supplement of bacteria (or cell) at the boundary $x=0$ to keep the flux of the bacteria being inward. A simple calculation  from \eqref{2.13} and equations in \eqref{1-6} yields
\begin{equation}\label{wave}
s=\left(\chi|\eta_-|\right)^{\frac{1}{3}},\ v_-=\left(\frac{\eta_-}{\chi^2}\right)^{\frac{1}{3}}<0, \ \ u_-=\left(\frac{\eta_-^2}{\chi}\right)^{\frac{1}{3}}.
\end{equation}
This implies that the wave profile \eqref{explicit} with \eqref{wave} restricted on $[0, \infty) $  is anticipated to be selected as the asymptotic profile of the initial-boundary value problem \eqref{ph}-\eqref{new-boundary data}. The rest of this paper will be devoted to proving this conjecture with some appropriate initial data. For convenience, in the sequel, we still use the notation $u_-$ and $v_-$, but keeping in mind they are uniquely determined by $\eta_-$ through \eqref{wave}.

\subsection{Set-up of initial data and statement of main results}
It is known in \cite{jin13} that if the initial function is a small perturbation of the traveling wave $(U,V)(x-st)$ in some suitable topological space, the solution of the Cauchy problem of \eqref{ph} will converge to a shifted traveling wave  $(U,V)(x-st-x_0)$ where the shift $x_0$ is determined by the initial data. However, for the system \eqref{ph} on the half space $[0,\infty)$,  a boundary layer may exist at the boundary $x=0$ because the boundary value of the traveling wave profile $(U,V)(x-st)$ at $x=0$ varies in time. Namely
$$(u-U)|_{x=0}=u(0,t)-U(-st)\ne 0$$
may occur since the value $u(0,t)$ is unknown due the non-zero flux boundary condition \eqref{new-boundary data}, see an illustration in Fig.\ref{fig}(a).
To control this boundary layer effect, we shall use the idea of Matsumura-Mei \cite{MM99} by shifting the traveling wave far away from the boundary initially with a shift $\b>0$ being a large constant, so that the boundary value of the shifted wave profile at $x=0$ is sufficiently close to $u_-$. By setting the initial datum  as a small perturbation of the shifted traveling wave profile $(U,V)(x-st-\b)$, the initial boundary value $u_0(0)$ will be close to $u_-$ as long as $\beta$ is sufficiently large (see Fig.\ref{fig}(b)). Then we anticipate that the resulting boundary value of the solution will asymptotically converge to $u_-$ as time tends to infinity due to the dissipation structure of the model. Accordingly we may expect that the time-asymptotic profile of the solution to \eqref{ph}-\eqref{new-boundary data} is $(U,V)(x-st+\a-\b)$ restricted on $[0,\infty)$ with another shift $\a$ to be determined.

\begin{figure}[!htbp]
\centering
\includegraphics[width=7.5cm]{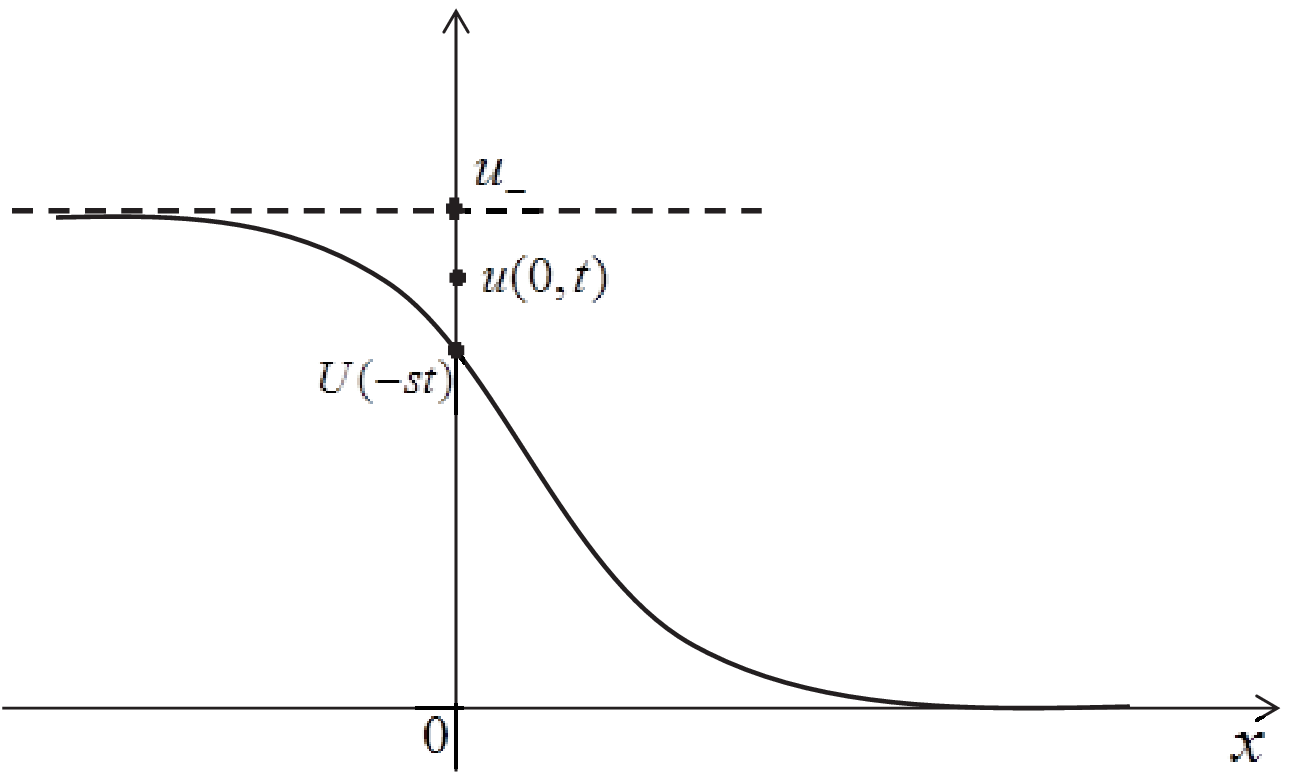}\hspace{0.5cm}
\includegraphics[width=7.5cm]{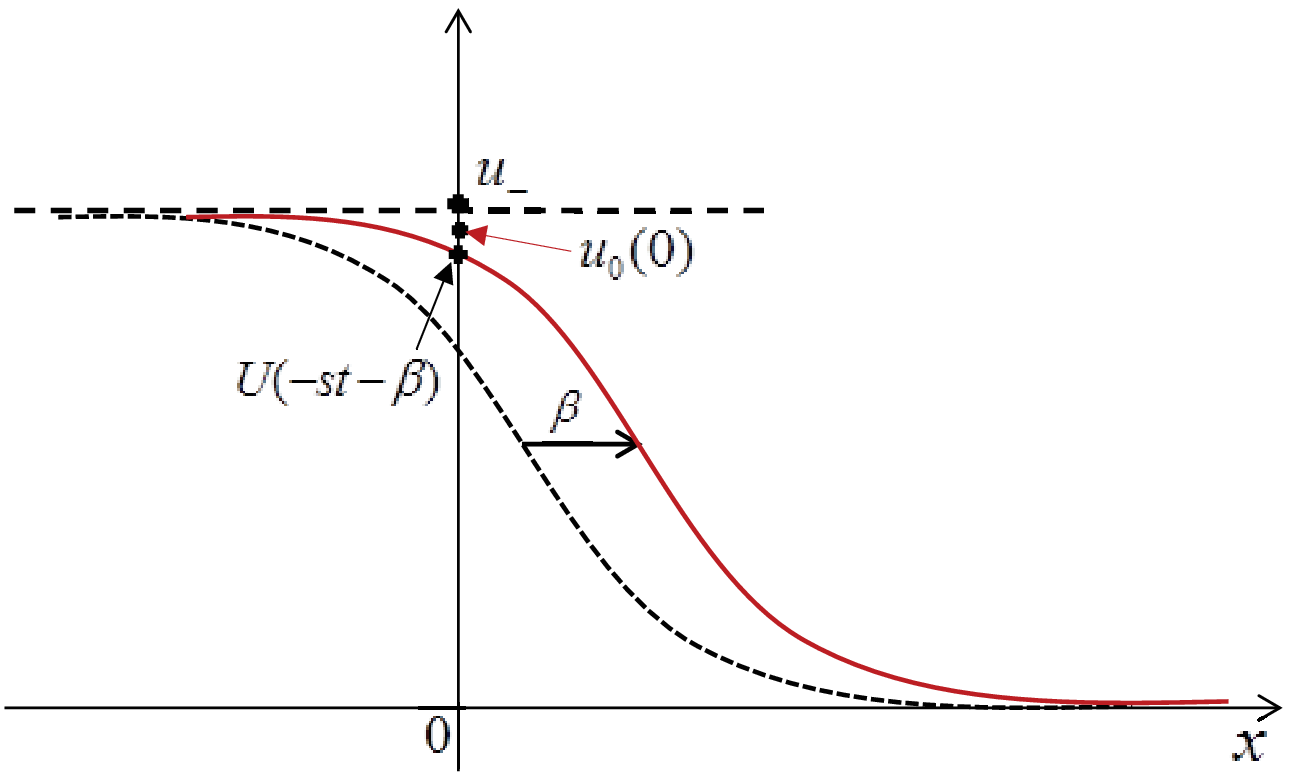}
(a) \hspace{7cm} (b)
\caption{Illustration of the boundary layer effect.}
\label{fig}
\end{figure}

The same transformed system \eqref{ph} with homogeneous Dirichlet boundary condition on $u$ was studied in a previous work \cite{Mei-peng-wang} and the convergence to traveling wave solutions restricted on $[0,\infty)$ was obtained only for the non-physical case $U(\infty)=u_+>0, V(-\infty)=v_-=0$ in the sense that the {\color{black} solution of \eqref{ph} was not able to give meaningful results to $c$ when it was passed} to the original chemotaxis system \eqref{omn} via (\ref{transformation}), see \cite[Remark 2.1]{Mei-peng-wang}. In this paper, we consider different boundary conditions on $(U,V)$ given by \eqref{boundary condition} directly derived from traveling wave profile $(U, \mathcal{C})$. Moreover the dynamic non-zero flux boundary condition \eqref{new-boundary data} gives no information on the boundary value of $u$ explicitly,   which is significantly different from \cite{Mei-peng-wang} and \cite{MM99} wherein the fixed boundary value was directly imposed. Hence extra boundary estimates are needed in this paper to prove the stability result.  More importantly, the result of \eqref{ph} with the non-zero flux boundary condition \eqref{new-boundary data} can now be successfully transferred to \eqref{omn}. Apart from these differences, we determine the shift $\a$ based on the first equation of \eqref{ph} instead of the second one of \eqref{ph} as used in \cite{Mei-peng-wang}, which enables us to derive our desired results with the non-zero flux boundary condition \eqref{new-boundary data}.

Below we shall briefly show how the shift $\alpha$ is determined and then state our main results of this paper. Indeed from the first equation of \eqref{ph}, we have
\[(u-U)_t=D(u-U)_{xx}+\k(uv-UV)_x.
\]
Integrating this equation over $\R_+$ with respect to $x$, by the boundary conditions \eqref{new-boundary data} and \eqref{2.13}, we get
\be\label{2.15}\begin{split}
&\frac{d}{dt}\ii[u(x,t)-U(x-st+\a-\b)]dx=[D(u-U)_{x}+\k(uv-UV)]|_0^\infty\\
&=-\eta(t)+(DU_{x}+\k UV)(-st+\a-\b)\\
&=-(\eta(t)-\eta_-)-\k u_-v_--sU(-st+\a-\b)\\
&=s(u_--U(-st+\a-\b))-(\eta(t)-\eta_-),
\end{split}\ee
where we have used the first equation of \eqref{3-4} and $s=-\k v_-$ owing to \eqref{1-6}. Integrating \eqref{2.15} in $t$ yields
\be\label{3.1}
\begin{split}
&\ii[u(x,t)-U(x-st+\a-\b)]dx\\&=\ii[u_0(x)-U(x+\a-\b)]dx+s\int_0^t(u_--U(-s\tau+\a-\b))d\tau-\int_0^t(\eta(\tau)-\eta_-)d\tau.
\end{split}\ee
To determine $\alpha$, we look for the solution satisfying $\ii[u(x,t)-U(x-st+\a-\b)]dx\rightarrow0$ as $t\rightarrow\infty$. Thus, if we set
\[I(\a):=\ii[u_0(x)-U(x+\a-\b)]dx\\+s\int_0^\infty(u_--U(-st+\a-\b))dt-\int_0^\infty(\eta(t)-\eta_-)dt,\]
then $I(\a)=0$. Differentiating $I(\a)$ in $\a$ gives
\[\begin{split} \frac{d I}{d\a}&=-\ii U'(x+\a-\b)dx -s\int_0^\infty U'(-st+\a-\b)dt\\&=U(\a-\b)+u_--U(\a-\b)\\&=u_-.\end{split}\]
Hence, it follows that
 \[\begin{split}0=I(\a)=I(0)+u_-\a=&\ii[u_0(x)-U(x-\b)]dx+s\int_0^\infty(u_--U(-st-\b))dt
\\&-\int_0^\infty(\eta(t)-\eta_-)dt+u_-\a,\end{split}\]
which enables us to determine $\a$ as
\[
\a=-\frac{1}{u_-}\left(\ii[u_0(x)-U(x-\b)]dx+\int_0^\infty\big[s(u_--U(-st-\b))-(\eta(t)-\eta_-)\big]dt\right).
\]
In view of \eqref{explicit}, one can easily calculate that
\[\int_0^\infty\big(u_--U(-st-\b))dt=u_-\int_0^\infty\frac{ e^{-\frac{s}{D}(st+\b)}}{ e^{-\frac{s}{D}(st+\b)}+1}dt
=\frac{Du_-}{s^2}\ln(1+ e^{-\frac{s\b}{D}}).\]
Thus, the formula of deriving the shift $\alpha$ reads
\be\label{alpha}
\begin{split}
\a=-\frac{1}{u_-}\ii[u_0(x)-U(x-\b)]dx-\frac{D}{s}\ln(1+ e^{-\frac{s\b}{D}})+\frac{1}{u_-}\int_0^\infty(\eta(t)-\eta_-)dt.
\end{split}\ee
Under \eqref{alpha},  by \eqref{3.1}, one verifies
\be\label{2.11}\begin{split}&\ii[u(x,t)-U(x-st+\a-\b)]dx\\&=I(\a)-s\int_t^\infty(u_--U(-s\tau+\a-\b))d\tau+\int_t^\infty(\eta(\tau)-\eta_-)d\tau\\
&=-s\int_t^\infty(u_--U(-s\tau+\a-\b))d\tau+\int_t^\infty(\eta(\tau)-\eta_-)d\tau\\
&=-\frac{Du_-}{s}\ln(1+ e^{\frac{s}{D}(-st+\a-\b)})+\int_t^\infty(\eta(\tau)-\eta_-)d\tau\\&\rightarrow0 \text{ as }t\rightarrow\infty,\end{split}\ee
as desired.
In view of \eqref{alpha}, the shift $\a$ is determined by the initial perturbation around the traveling wave $U(x-\b)$. With the anti-derivative technique {\color{black} usually} used for the conservation laws (cf. \cite{smoller}), one is motivated to define the initial perturbation by
\begin{equation}\label{ict}
(\Phi_0, \Psi_0)(x)\triangleq-\int_x^{\infty}(u_0(y)-U(y-\b),v_0(y)-V(y-\b))dy.
\end{equation}
We also assume that $\eta(t)$ is a small perturbation of $\eta_-$ in the sense of
\begin{equation}\label{around}
\int_0^t|\eta_\tau(\tau)|d\tau+|\eta(t)-\eta_-|+\int_0^\infty|\eta(t)-\eta_-|dt+\int_0^t\int_\tau^\infty|\eta(z)-\eta_-|d\tau\leq \delta,\ \forall t\geq0
\end{equation}
for some small constant $\delta>0$.
Typical candidates of $\eta(t)$ include functions satisfying $|\eta_t(t)|+|\eta(t)-\eta_-|\leq \delta(1+t)^{-k} \text{ for } k>2$. \\

We are now ready to state our main results as follows.

\begin{theorem}\label{thm-1} Assume that $v_+=0$, $\eta_-<0$, and that $\eta(t)$ satisfies \eqref{around}. Let $(U,V)$ be a traveling wave  of $\eqref{ph}$ satisfying \eqref{boundary condition} with $u_-$, $v_-$ given by \eqref{wave}. There exists a constant $\varepsilon_0>0$ such that if $\norm{\Phi_0}_{2,w_0}+\norm{\Psi_0}_{2}+\norm{\Psi_{0x}}_{1,w_0}+\delta+\beta^{-1}\leq \varepsilon_0$  and $(\|\Phi_0\|_1+\delta)\b\leq1$,
then the initial-boundary value problem $\eqref{ph}$-$\eqref{new-boundary data}$ with $x \in [0,\infty)$ has a unique global solution $(u,v)(x,t)$ satisfying
\begin{equation}\label{regularity}
\begin{cases}
u(x,t)-U(x-st+\a-\b)\in C([0,\infty); H^1_w)\cap L^2((0,\infty);H^2_w),\\
v(x,t)-V(x-st+\a-\b)\in C([0,\infty); H^1_w)\cap L^2((0,\infty);H^1_w),
\end{cases}\end{equation}
where $\a=\a(\b)$ is a shift determined by \eqref{alpha}, and the weight function $w$ is defined by
\begin{equation}\label{weight fucntion}
w(x,t):=1+e^{\frac{s}{D}(x-st+\a-\b)} \text{ with } w_0(x):=w(x,0).
\end{equation}
Furthermore, the solution has the following asymptotic profile
\begin{equation}\label{asym}
\sup\limits_{x\in\R_+}\abs{(u,v)(x,t)-(U,V)(x-st+\a-\b)}\to 0~~\mathrm{as}~~ t\to+\infty.
\end{equation}
\end{theorem}

Base on the Cole-Hopf transformation \eqref{transformation}, we are able to transfer the stability results back to the original chemotaxis model \eqref{omn} with \eqref{omn-init}-\eqref{omn-bound}.
\begin{theorem}\label{mainth2}
Assume that $\eta(t)$ satisfies \eqref{around} and that $\eta_-<0$, $c_+>0$. Let  $(U,\mathcal{C})$ be a traveling wave profile of \eqref{omn} satisfying \eqref{property} with $u_-$ given by \eqref{wave}. Then there exists a constant $\varepsilon_0>0$
such that if   $\norm{\Phi_0}_{2,w_0}+\norm{\Psi_0}_{2}+\norm{\Psi_{0x}}_{1,w_0}+\delta+\beta^{-1}\leq \varepsilon_0$ and $(\|\Phi_0\|_1+\delta)\b\leq1$, where
\[\Phi_0(x)=-\int_x^{\infty}(u_0(y)-U(y-\b))dy,\  \Psi_0(x)=-\frac{1}{\mu}(\ln c_0(x)-\ln \mathcal{C}(x-\b)),
\]
then the  initial-boundary value problem \eqref{omn} with $x\in [0,\infty)$ and \eqref{omn-init}-\eqref{omn-bound} has a unique global solution $(u,c)(x,t)$ satisfying
\begin{equation*}
\begin{split}
&(u,{c_x}/{c})(x,t)-(U,{\mathcal{C}_x}/{\mathcal{C}})(x-st+\a-\b) \in C([0,\infty);H_w^1) \cap L^2((0,\infty);H_w^1),\\
&c(x,t)-\mathcal{C}(x-st+\a-\b) \in C([0,\infty);H^2),
\end{split}
\end{equation*}
and possessing the following asymptotic profile:
\begin{equation*}
\sup \limits_{x \in \R_+} |(u, c )(x,t)-(U, \mathcal{C})(x-st+\a-\b)| \to 0 \ \
\mathrm{as} \ \ t \to \infty.
\end{equation*}
\end{theorem}

\section{Proofs of the main results}
\subsection{Reformulation of the problem}In this section, we will first prove Theorem \ref{thm-1} for the transformed problem $\eqref{ph}$-$\eqref{new-boundary data}$ with $(x,t)\in \R_+\times \R_+$ by the weighted energy method, and then prove Theorem \ref{mainth2} by passing the results to the original chemotaxis model \eqref{omn} with \eqref{omn-init}-\eqref{omn-bound} under the Cole-Hopf transformation \eqref{transformation}.
By \eqref{2.11}, one sees that the shift $\a$ is selected such that $\ii[u(x,t)-U(x-st+\a-\b)]dx\rightarrow0$ as $t\rightarrow+\infty$. Then we
employ the technique of anti-derivative to study the asymptotic behavior of solutions  to  $\eqref{ph}$-$\eqref{new-boundary data}$ and define
\begin{eqnarray*}
\begin{aligned}
(\phi(x,t), \psi(x,t))\triangleq-\int_x^{\infty}(u(y,t)-U(y-st+\a-\b),v(y,t)-V(y-st+\a-\b))dy.\\
\end{aligned}
\end{eqnarray*}
That is
\begin{equation}\label{2-5}
(u,v)(x,t)=(U,V)(x-st+\a-\b)+(\phi_x,\psi_x)(x,t).
\end{equation}
Then by  \eqref{ph} and  \eqref{2.11}, $(\phi,\psi)$ satisfies
\begin{equation}\label{nonlinear system}
\begin{cases}
\phi_t=D\phi_{xx}+\chi V\phi_x+\chi U\psi_x+\chi \phi_x\psi_x,~~t>0, \ x\in\R_+,\\
\psi_t=\phi_x,
\end{cases}
\end{equation}
with initial  condition
\be\label{3.3}
(\phi_0, \psi_0)(x) =-\int_x^{\infty}(u_0(y)-U(y+\a-\b),v_0(y)-V(y+\a-\b))dy,
\ee
and boundary condition
\be\label{3.6}
\begin{split}
\phi(0,t)=-s\int_t^\infty(u_--U(-s\tau+\a-\b))d\tau+\int_t^\infty(\eta(\tau)-\eta_-)d\tau\triangleq A(t).
\end{split}\ee
Note that there are two distinguishing differences from the stability of traveling wave solutions in the whole space $\R$ established in the previous works (e.g. see \cite{jin13, Li09}). First the initial perturbation here is not required to be of zero integral. Second the boundary value is dynamic (depending on time).
The solution space of the reformulated problem \eqref{nonlinear system}-\eqref{3.6} is
\begin{equation*}
\begin{split}
X(0,T):=\{&(\phi(x,t),\psi(x,t))\big|\phi\in C([0,T]; H^2_w),\phi_x\in L^2((0,T);H^2_w)\\&\psi\in C([0,T]; H^2),\psi_x\in C([0,T]; H^1_w)\cap L^2((0,T);H^1_w)\},
\end{split}\end{equation*}
for $T\in(0,+\infty]$, where the weight function $w$ is defined by $\eqref{weight fucntion}$. Set
\begin{equation*}
N(t)\triangleq\sup_{\tau\in[0,t]}(\|\psi(\cdot,\tau)\|+\|\psi_x(\cdot,\tau)\|_{1,w}+\|\phi(\cdot,\tau)\|_{2,w}).
\end{equation*}
Clearly, if $\phi\in H^2_w$, then $\phi\in H^2$ since $w\geq1$. Thus the Sobolev embedding theorem yields
\begin{equation}\label{Sobolev}
\sup_{\tau\in[0,t]}\{\|\phi(\cdot,\tau)\|_{L^\infty},\|\phi_x(\cdot,\tau)\|_{L^\infty},\|\psi(\cdot,\tau)\|_{L^\infty},
\|\psi_x(\cdot,\tau)\|_{L^\infty}\}\leq N(t).
\end{equation}

For the reformulated problem \eqref{nonlinear system}-\eqref{3.6}, we shall prove the following results.
\begin{theorem}\label{global existence}
There exists a positive constant $\varepsilon_1$, such that if $N(0)+\b^{-1}+\delta\leq\varepsilon_1$, then the initial-boundary value problem \eqref{nonlinear system}-\eqref{3.6} has a unique global solution $(\phi,\psi)\in X(0,\infty)$ such that
\begin{equation}\label{priori}
\begin{split}
\|\phi\|_{2,w}^2+\|\psi\|^2&+\|\psi_x\|_{1,w}^2+\int_0^t(\|\phi_x(\tau)\|_{2,w}^2+\|\psi_x(\tau)\|_{1,w}^2)d\tau\\
&\leq C\left(\|\phi_0\|_{2,w_0}^2+\|\psi_0\|^2+\|\psi_{0x}\|_{1,w_0}^2+e^{-\frac{s\b}{D}}+\delta\right)\leq C (N^2(0)+e^{-\frac{s\b}{D}}+\delta)
\end{split}
\end{equation}
for any $t\in [0,\infty)$, where $w_0:=w(x,0)$. Moreover, it holds that
\begin{equation}\label{long-time behavior}
\sup\limits_{x\in\R^+}\abs{(\phi_x,\psi_x)(x,t)}\to 0~~as~~t\to\infty.
\end{equation}
\end{theorem}

To apply the results of Theorem \ref{global existence} to the problem \eqref{ph}-\eqref{new-boundary data}, we need to further clarify the relation between the initial data $(\phi_0,\psi_0)$ and $(\Phi_0,\Psi_0)$ since they not exactly the same. In view of \eqref{alpha}, the shift $\a=\a(\b)$ is a function of $\b$ and the asymptotic behavior of $\a$ can be characterized as follows.
\begin{lemma}\label{lem1}
If $\|\Phi_0\|_1+\delta+\b^{-1}\rightarrow0$, then $\a\rightarrow0$.
\end{lemma}
\begin{proof} By \eqref{alpha}, one has
\[\a=\frac{1}{u_-}\Phi_0(0)-\frac{D}{s}\ln(1+\kappa e^{-\frac{s\b}{D}})+\frac{1}{u_-}\int_0^\infty(\eta(t)-\eta_-)dt.\]
Since $|\Phi_0(0)|\leq \|\Phi_0\|_1$, and $\ln(1+ e^{-\frac{s\b}{D}})\leq e^{-\frac{s\b}{D}}$, it follows from \eqref{around} that
\begin{equation}\label{3.8}
|\a|\leq C (\|\Phi_0\|_1+e^{-\frac{s\b}{D}}+\delta)\rightarrow0,\end{equation}
as $\|\Phi_0\|_1+\delta+\b^{-1}\rightarrow0$.\end{proof}

We now present the relation between $(\phi_0,\psi_0)$ and $(\Phi_0,\Psi_0)$.
\begin{lemma}\label{lem2}
Let $(\|\Phi_0\|_1+\delta)\b$ be bounded for all $\beta>0$. If $\|\Phi_0\|_{2,w_0}+\|\Psi_0\|+\|\Psi_{0x}\|_{1,w_0}+\delta+\b^{-1}\rightarrow0$, then $\|\phi_0\|_{2,w_0}+\|\psi_0\|+\|\psi_{0x}\|_{1,w_0}\rightarrow0$.
\end{lemma}
\begin{proof} It first follows from \eqref{3.3} and \eqref{ict} that
\[\begin{split}\phi_0(x) &=-\int_x^{\infty}(u_0(y)-U(y+\a-\b))dy\\
&=\Phi_0(x)+\int_x^{\infty}(U(y+\a-\b)-U(y-\b))dy\\
&=\Phi_0(x)+\int_x^{\infty}\int_0^\a U'(y+\theta-\b)d\theta dy\\
&=\Phi_0(x)-\int_0^\a U(x+\theta-\b)d\theta\\
&\triangleq\Phi_0(x)+B(x).\end{split}\]
By Lemma  \ref{lem1}, $|\a|\ll1$. Thus, owing to \eqref{explicit}, we have
$|B(x)|\leq\frac{|\a|u_-}{ e^{\frac{s}{D}(x-1-\b)}+1}.$
This gives us that
\[\begin{split}
\|B\|_{w_0}^2&\leq\a^2u_-^2\int_0^\infty\frac{e^{\frac{s}{D}(x+1-\b)}+1}{( e^{\frac{s}{D}(x-1-\b)}+1)^2}dx\\
&=\a^2u_-^2\left(\int_0^{\b}\frac{e^{\frac{s}{D}(x+1-\b)}+1}{( e^{\frac{s}{D}(x-1-\b)}+1)^2}dx+\int_{\b}^\infty\frac{ e^{\frac{s}{D}(x+1-\b)}+1}{( e^{\frac{s}{D}(x-1-\b)}+1)^2}dx\right)\\
&\leq\a^2u_-^2\left(\int_0^{\b}e^{\frac{s}{D}(x+1-\b)}dx+\b+\int_{\b}^\infty\frac{e^{\frac{s}{D}(x+1-\b)}}{e^{\frac{2s}{D}(x-1-\b)}}dx+\int_{\b}^\infty\frac{1}{ e^{\frac{2s}{D}(x-1-\b)}}dx\right)\\
&=\a^2u_-^2 s^{-1}D e^{\frac{s}{D}}\left(1-e^{-\frac{s}{D}\b}+e^{\frac{2s}{D}}+e^{\frac{s}{D}}\right)+\a^2\b u_-^2\\
&\leq C\a,
\end{split}\]
where we have used the estimate \eqref{3.8} in the last inequality. Similarly, we have $\|B\|_{2,w_0}^2\leq C\a$ for some $C>0$. By Lemma \ref{lem1}, we then have
\[\|\phi_0\|_{2,w_0}^2\leq\|\Phi_0\|_{2,w_0}^2+\|B\|_{2,w_0}^2\rightarrow0,\]
as  $\|\Phi_0\|_{2,w_0}+\delta+\b^{-1}\rightarrow0$. Similarly, $\|\psi_0\|+\|\psi_{0x}\|_{1,w_0}\rightarrow0$
as  $\|\Psi_0\|+\|\Psi_{0x}\|_{1,w_0}+\delta+\b^{-1}\rightarrow0$.
\end{proof}

Clearly Theorem \ref{thm-1} is a direct consequence of Theorem \ref{global existence} and Lemma \ref{lem2}. In the remaining part of this section, we will focus on the proof of Theorem \ref{global existence}, which follows from the local existence theorem and the {\it a priori} estimate given below.
\begin{proposition}[Local existence]\label{local existence}
 For any $\varepsilon_0>0$, there exists a positive constant $T_0$ depending on $\varepsilon_0$ such that if $(\phi_0,\psi_0)\in H^2_{w_0}$ with $N(0)+\delta+\b^{-1}\leq \varepsilon_0$, then  the problem \eqref{nonlinear system}-\eqref{3.6} has a unique solution $(\phi,\psi)\in X(0,T_0)$ satisfying $N(t)\leq 2\varepsilon_0$ for any $0\leq t\leq T_0$.
\end{proposition}
\begin{proposition}[A priori estimate]\label{a priori estimate} Assume that $(\phi,\psi)\in X(0,T)$ is a solution obtained in Proposition $\ref{local existence}$ for a positive constant $T$. Then there is a positive constant $\varepsilon_2>0$, independent of $T$, such that if
\begin{equation*}
N(t)\leq \varepsilon_2
\end{equation*}
for any $0\leq t\leq T$, then the solution $(\phi,\psi)$ of \eqref{nonlinear system}-\eqref{3.6} satisfies $\eqref{priori}$ for any $0\leq t\leq T$.
\end{proposition}

The local existence in Proposition \ref{local existence} can be proved using the standard iteration method (see \cite{Nishida78}), and the details will be omitted for brevity. Now it remains to derive the {\it a priori}  estimates in Proposition \ref{a priori estimate}.  Without loss of generality, we assume that $N(t)\ll1$, $|\a|\ll1$ and $\b\gg1$ in what follows.\\

We first derive the basic $L^2$-estimate.

\begin{lemma}\label{Basic L^2 estimate}Let the assumptions of Proposition \ref{a priori estimate} hold. Then there exists a constant $C>0$ such that
\begin{equation}\label{L^2 estimate}
\norm{\psi}^2+\norm{\phi}^2_w+\int_0^t\norm{\phi_x}_w^2\leq C\big(\|\psi_0\|^2+\|\phi_0\|^2_{w_0}+e^{-\frac{s\b}{D}}+\delta\big)+CN(t)\int_0^t\int_0^\infty\frac{\psi_x^2}{U}.
\end{equation}
\end{lemma}
\begin{proof}
Multiplying $\eqref{nonlinear system}_1$ by $\phi/U$ and $\eqref{nonlinear system}_2$ by $\chi\psi$,  and adding them, we obtain
\[\frac{1}{2}\left(\frac{\phi^2}{U}+\chi\psi^2\right)_t+\frac{D\phi_x^2}{U}+\frac{\phi^2}{2}\left(\frac{s+\chi V}{U}\right)_x=\left(\frac{D\phi\phi_x}{U}+\frac{\k V\phi^2}{2U}+\k\phi\psi\right)_x+\frac{DU_x\phi\phi_x}{U^2}+\chi\frac{\phi\phi_x\psi_x}{U}.\]
From Young's inequality: $\frac{DU_x\phi\phi_x}{U^2}\leq\frac{D\phi_x^2}{2U}+\frac{DU_x^2\phi^2}{2U^3}$,
it follows that
\be\begin{split}\label{3.11}
&\frac{1}{2}\left(\frac{\phi^2}{U}+\chi\psi^2\right)_t+\frac{D\phi_x^2}{2U}+\frac{\phi^2}{2}\left[\left(\frac{s+\chi V}{U}\right)_x-\frac{DU_x^2}{2U^3}\right]\\&\leq\left(\frac{D\phi\phi_x}{U}+\frac{\k V\phi^2}{2U}+\k\phi\psi\right)_x+\chi\frac{\phi\phi_x\psi_x}{U}.\end{split}\ee
By \eqref{3-4}, it is easy to see that
$$
\left(\frac{s+\chi V}{U}\right)_x-\frac{DU_x^2}{2U^3}=-\frac{\k}{s}\cdot\frac{U_x}{U}>0.
$$
Thus, integrating \eqref{3.11} over $[0,\infty)\times[0,t]$, we derive
\begin{equation}\label{3.13}
\begin{split}
&\frac{1}{2}\ii\left(\frac{\phi^2}{U}+\chi\psi^2\right)+\frac{D}{2}\int_0^t\ii\frac{\phi_x^2}{U}\\&\leq\frac{1}{2}\ii\left(\frac{\phi_0^2}{U}+\chi\psi_0^2\right)-\int_0^t\left(\frac{D\phi\phi_x}{U}+\frac{\k V\phi^2}{2U}+\k\phi\psi\right)(0,\tau)d\tau+\chi\int_0^t\ii\frac{\phi_x\psi_x\phi}{U}.
\end{split}\end{equation}
By \eqref{Sobolev} and Young's inequality, the last term in \eqref{3.13} is estimated as
\begin{equation}\label{3-3}
\chi\int_0^t\ii\frac{\phi_x\psi_x\phi}{U}\leq\frac{N(t)D}{4}\int_0^t\ii\frac{\phi_x^2}{U}
+\frac{N(t)\k^2}{D}\int_0^t\ii\frac{\psi_x^2}{U}.\end{equation}
The boundary term in \eqref{3.13} can be estimated as follows. 
With the fact $\ln(1+x)\leq x$ for all $x\geq 0$, one has
$$s\int_t^\infty(u_--U(-s\tau+\a-\b))d\tau=\frac{Du_-}{s}\ln(1+ e^{\frac{s}{D}(-st+\a-\b)})\leq\frac{Du_- }{s} e^{\frac{s}{D}(-st+\a-\b)}.$$
Thus, by \eqref{3.6} and \eqref{around}, we have
\be\label{3.14}\int_0^t|\phi(0,\tau)|\leq C\int_0^te^{-\frac{s}{D} (s\tau+\b)}d\tau+\int_0^t\int_\tau^\infty|\eta(z)-\eta_-|dzd\tau\leq C(e^{-\frac{s\b}{D}}+\delta).
\ee
In addition, by Lemma \ref{etw}, it holds that $U(-st+\a-\b)>U(0)=\frac{u_-}{2}$ and $|V(-st+\a-\b)|\leq |v_-|$. Hence
\begin{equation}\label{3-5}
\left|\int_0^t\left(\frac{D\phi\phi_x}{U}+\frac{\k V\phi^2}{2U}+\k\phi\psi\right)(0,\tau)d\tau\right|\leq CN(t)\int_0^t|\phi(0,\tau)|d\tau\leq C(e^{-\frac{s}{D} \b}+\delta),
\end{equation}
where we have used $\|\phi(\cdot,t)\|_{L^\infty},\|\phi_x(\cdot,t)\|_{L^\infty},\|\psi(\cdot,t)\|_{L^\infty}\leq N(t)$, see \eqref{Sobolev}. Substituting \eqref{3-3} and \eqref{3-5} into \eqref{3.13}, and noting the fact that
\be\label{3.15}
C_1w\leq\frac{1}{U}\leq C_2w ,\ee
one gets  $\eqref{L^2 estimate}$ immediately and the proof is completed.
\end{proof}

We next present the estimate of the first order derivatives of $(\phi,\psi)$.
\begin{lemma}\label{H^1 estimate}Let the assumptions of Proposition \ref{a priori estimate} hold. Then
\begin{equation}\label{eqn-2.25}
\begin{split}
&\|\psi\|_1^2+\|\phi\|^2_{1,w}
+\|\psi_x\|_w^2+\int_0^t\left(\|\phi_x\|^2_{1,w}
+\|\psi_x\|_w^2\right)
\leq C\left(\|\psi_{0x}\|_{w_0}^2+\|\phi_0\|^2_{1,w_0}
+\|\psi_0\|_1^2+\delta+e^{-\frac{s\b}{D}}\right)
\end{split}\end{equation}
holds for some constant $C>0$.
\end{lemma}

\begin{proof} The proof is divided into three steps.

{\bf Step 1}. Weighted energy estimate. Multiplying $\eqref{nonlinear system}_1$ by $\phi_t/U$, noting that
\[\begin{split}\frac{D\phi_t\phi_{xx}}{U}&=\left(\frac{D\phi_t\phi_{x}}{U}\right)_x-\left(\frac{D\phi_t}{U}\right)_x\phi_x\\
&=\left(\frac{D\phi_t\phi_{x}}{U}\right)_x-\left(\frac{D\phi_{x}^2}{2U}\right)_t+\frac{DsU_x}{2U^2}\phi_x^2+\frac{DU_x}{U^2}\phi_t\phi_x,\\
\chi\phi_{t}\psi_x&=\chi(\phi\psi_x)_t-\chi\phi\psi_{xt}=\chi(\phi\psi_x)_t-\chi\phi\phi_{xx}\\&=\chi(\phi\psi_x)_t-\chi(\phi\phi_x)_x+\chi\phi_x^2,\end{split}\]
where we have used the equation
\be\label{2}\psi_{tx}=\phi_{xx},\ee
we then obtain
\[\begin{split}\label{3.16}
\left(\frac{D\phi_{x}^2}{2U}\right)_t+\frac{\phi_{t}^2}{U}-\frac{DsU_x}{2U^2}\phi_x^2
=&\left(\frac{D\phi_t\phi_x}{U}-\chi\phi\phi_x\right)_x+\chi(\phi\psi_x)_t+\chi\phi_x^2+\frac{DU_x}{U^2}\phi_t\phi_x\\
&+\frac{\chi V}{U}\phi_t\phi_x+\frac{\chi \phi_x\psi_x\phi_t}{U}.
\end{split}\]
Integrating this equation over $[0,+\infty)\times[0,t]$ along with Young's inequality which gives
\[\frac{\chi V}{U}\phi_t\phi_x\leq\frac{\phi_{t}^2}{4U}+\frac{\chi^2V^2\phi_x^2}{U},\ \frac{DU_x}{U^2}\phi_t\phi_x\leq\frac{\phi_{t}^2}{4U}+\frac{D^2U_x^2\phi_x^2}{U^3},\]
and noting that $U_x<0$, we have
\begin{equation}\label{3.18}
\begin{split}
&\frac{D}{2}\ii\frac{\phi_x^2}{U}+\frac{1}{2}\int_0^t\ii\frac{\phi_t^2}{U}+\frac{Ds}{2}\int_0^t\ii\frac{|U_x|\phi_x^2}{U}\\&
\leq\frac{D}{2}\ii\frac{\phi_{0x}^2}{U}+\int_0^t\left(\chi\phi-\frac{D\phi_t}{U}\right)\phi_x(0,\tau)d\tau+\chi\ii\phi\psi_x-\chi\ii\phi_0\psi_{0x}\\&\quad
+\int_0^t\ii\left(\chi+\frac{\chi^2V^2}{U}+\frac{D^2U_x^2}{U^3}\right)\phi_x^2+\chi\int_0^t\ii\frac{\phi_x\psi_x\phi_t}{U}.
\end{split}\end{equation}
By Young's inequality, one has
\begin{equation}\label{3.22}
\begin{split}
&\chi\ii\phi\psi_x\leq\frac{\epsilon}{4}\ii\psi_x^2+\frac{\chi^2}{\epsilon}\ii\phi^2,\\
&\chi\int_0^t\ii\frac{\phi_x\psi_x\phi_t}{U}\leq N(t)\int_0^t\ii\frac{\phi_t^2}{4U}+\chi^2N(t)\int_0^t\ii\frac{\psi_x^2}{U},
\end{split}\end{equation}
where $\epsilon>0$ is a small constant to be determined later, and we have used  $\|\phi_x(\cdot,t)\|_{L^\infty}\leq N(t)$.
To estimated the boundary term, noting that by \eqref{3.6},
\begin{equation}\label{At}
A'(t)=s(u_--U(-st+\a-\b))-(\eta(t)-\eta_-),\end{equation}
we have
\[\begin{split}
|\phi_t(0,t)|=|A'(t)|\leq\frac{s u_- e^{\frac{s}{D}(-st+\a-\b)}}{1+ e^{\frac{s}{D}(-st+\a-\b)}}+|\eta(t)-\eta_-|.\end{split}\]
It then follows from \eqref{3.14} and \eqref{around} that
\be\label{3.21}\begin{split}
\left|\int_0^t\left(\chi\phi-\frac{D\phi_t}{U}\right)\phi_x(0,\tau)d\tau\right|\leq CN(t)\int_0^t\left(|\phi|+|\phi_t|\right)(0,\tau)d\tau\leq C(e^{-\frac{s}{D} \b}+\delta).\end{split}\ee
Thus, substituting \eqref{3.22} and \eqref{3.21} into \eqref{3.18}, by Lemma \ref{Basic L^2 estimate}, and observing that
\begin{equation}\label{3-6}
\frac{|U_x|}{U}=\frac{(s+\k V)}{D}\leq\frac{s}{D},
\end{equation}
we get
\begin{equation}\label{3.20}
\begin{split}
&\ii\frac{\phi_x^2}{U}+\int_0^t\ii\frac{\phi_t^2}{U}\\
&\leq C\left(\|\psi_0\|_1^2+\|\phi_0\|^2_{1,w_0}+e^{-\frac{s}{D} \b}+\delta\right)+\epsilon\ii\psi_x^2+CN(t)\int_0^t\ii\frac{\psi_x^2}{U}.
\end{split}\end{equation}

{\bf Step 2}. Elliptic estimate.
The first equation of $\eqref{nonlinear system}$ gives
\begin{equation*}\label{eqn-2.18}
D\phi_{xx}+\chi U\psi_x=\phi_{t}-\chi V\phi_x-\chi\phi_x\psi_x.
\end{equation*}
Taking square and dividing the above equation by $U$ leads to
\begin{equation}\label{3.23}
\frac{D^2\phi_{xx}^2}{U}+\chi^2 U\psi_x^2+2D\chi\phi_{xx}\psi_x\leq2\left(\frac{\phi_t^2}{U}+\frac{\chi^2 V^2\phi_x^2}{U}+\frac{\chi^2 \phi_x^2\psi_x^2}{U}\right).
\end{equation}
Owing to \eqref{2}, $2D\chi\phi_{xx}\psi_x=D\k(\psi_x)_t$. Then integrating \eqref{3.23} over $[0,\infty)\times[0,t]$, by \eqref{3.20} and Lemma \ref{Basic L^2 estimate},
we obtain
\begin{equation*}
\begin{split}
&D^2\int_0^t\ii\frac{\phi_{xx}^2}{U}+\chi^2\int_0^t\ii U\psi_x^2+ D\chi\ii\psi_x^2\\
&\leq C\left(\|\psi_0\|_1^2+\|\phi_0\|^2_{1,w_0}+e^{-\frac{s}{D} \b}+\delta\right)+2\epsilon\ii\psi_x^2+CN(t)\int_0^t\ii\frac{\psi_x^2}{U},
\end{split}
\end{equation*}
where we have used $\norm{\phi_x(\cdot,t)}_{L^\infty}\leq N(t)$. Now choosing $\epsilon=D\chi/4$, we get
\begin{equation}\label{eqn-5.19}
\begin{split}
&\ii\psi_x^2+\int_0^t\ii U\psi_x^2+\int_0^t\ii\frac{\phi_{xx}^2}{U}\\
&\leq C\left(\|\psi_0\|_1^2+\|\phi_0\|^2_{1,w_0}+e^{-\frac{s}{D} \b}+\delta\right)+CN(t)\int_0^t\ii\frac{\psi_x^2}{U}.
\end{split}
\end{equation}It further follows from \eqref{3.20} that
\begin{equation}\label{3.24}
\begin{split}
\ii\frac{\phi_x^2}{U}+\int_0^t\ii\frac{\phi_t^2}{U}\leq C\left(\|\psi_0\|_1^2+\|\phi_0\|^2_{1,w_0}+e^{-\frac{s}{D} \b}+\delta\right)+CN(t)\int_0^t\ii\frac{\psi_x^2}{U}.
\end{split}\end{equation}

{\bf Step 3}. To complete the proof of $\eqref{eqn-2.25}$, it remains to estimate $\int_0^t\ii\frac{\psi_x^2}{U}$. By \eqref{3.15}, it suffices to estimate $\int_0^t\ii w\psi_x^2$. Because $U$ is monotone decreasing in $(-\infty, \infty)$, it holds that $\frac{u_-}{2}=U(0)<U(z)<u_-$ for all $z \in (-\infty, 0)$. In addition, $1<w(x,t)<2$ for all $x \in (0,st-\a+\b)$. Thus,  $U(x-st+\a-\b) > \frac{u_-}{4}w(x,t)$ for all $x \in(0,st-\a+\b)$. Then it follows from \eqref{eqn-5.19}  that
\begin{equation}\label{eqn-5.21}
\begin{split}
&\int_0^{st-\a+\b}w\psi_x^2+\int_0^t\int_0^{s\tau-\a+\b} w \psi_x^2\\&\leq C\left(\|\psi_0\|_1^2+\|\phi_0\|^2_{1,w_0}
+e^{-\frac{s}{D} \b}+\delta+N(t)\int_0^t\int\frac{\psi_x^2}{U}\right).
\end{split}
\end{equation}
We then multiply  \eqref{2} by $w\psi_x$ to obtain
\begin{equation*}
\psi_{xt}w\psi_x=w\psi_x\phi_{xx},
\end{equation*}
which leads to
\begin{equation}\label{eqn-2.21}
\left(\frac{w\psi_x^2}{2}\right)_t-\frac{w_t\psi_x^2}{2}=w\psi_x\phi_{xx}.
\end{equation}
Now integrating \eqref{eqn-2.21} over $[st-\a+\b,+\infty)\times[0,t]$, and using the fact that
\begin{equation}\label{wt}
-w_t=\frac{s^2}{D}e^{\frac{s}{D}(x-st+\a-\b)}\geq\frac{s^2w}{2D}\ \text{ for } x \in(st-\a+\b,+\infty),
\end{equation} we have
\begin{equation*}
\begin{split}
&\frac{1}{2}\int_{st-\a+\b}^{\infty} w\psi_x^2+\frac{s^2}{4D}\int_0^t\int_{s\tau-\a+\b}^{\infty}w\psi_x^2\\&\leq\frac{1}{2}\ii w_0\psi_{0x}^2+\int_0^t\int_{s\tau-\a+\b}^{\infty}w\psi_x\phi_{xx}\\
&\leq \frac{1}{2}\|\psi_{0x}\|_{w_0}^2+\frac{s^2}{8D}\int_0^t\int_{s\tau-\a+\b}^{\infty}w\psi_x^2
+\frac{2D}{s^2}\int_0^t\int_{s\tau-\a+\b}^{\infty}w\phi_{xx}^2,
\end{split}
\end{equation*}
which in combination with \eqref{3.15}, \eqref{eqn-5.19} and \eqref{eqn-5.21} gives
\begin{equation*}
\ii w\psi_x^2+\int_0^t\ii w\psi_x^2
\leq C\left(\|\psi_{0x}\|_{w_0}^2+\|\phi_0\|^2_{1,w_0}
+\|\psi_0\|_1^2+e^{-\frac{s}{D} \b}+\delta+N(t)\int_0^t\ii w\psi_x^2\right).
\end{equation*}
Thus, it holds that
\begin{equation}\label{eqn-2.24}
\ii w\psi_x^2+\int_0^t\ii w\psi_x^2
\leq C\left(\|\psi_{0x}\|_{w_0}^2+\|\phi_0\|^2_{1,w_0}
+\|\psi_0\|_1^2+e^{-\frac{s}{D} \b}+\delta\right),
\end{equation}
since $N(t)$ is small enough. The desired estimate $\eqref{eqn-2.25}$ follows from \eqref{L^2 estimate}, \eqref{eqn-5.19}, \eqref{3.24} and \eqref{eqn-2.24}.
\end{proof}

We now derive the estimates of the second order derivative of $(\phi,\psi)$.

\begin{lemma}\label{sec}
Let the assumptions of Proposition \ref{a priori estimate} hold. Then there exists a constant $C>0$ such that
\begin{equation}\label{eqn-2.31}
\begin{split}
&\|\phi_{xx}\|^2_{1,w}
+\|\psi_{xx}\|_w^2+\int_0^t\left(\|\phi_{xxx}\|^2_w
+\|\psi_{xx}\|_w^2\right)\\&
\leq C\left(\|\psi_{0x}\|_{1,w_0}^2+\|\phi_0\|^2_{2,w_0}
+\|\psi_0\|^2+e^{-\frac{s\b}{D}}+\delta\right).\end{split}
\end{equation}
\end{lemma}
\begin{proof}
 We differentiate \eqref{nonlinear system} with respect to $x$ to get
\begin{equation}\label{eqn-2.15}
\left\{\begin{array}{ll}
\phi_{xt}=D\phi_{xxx}+\chi U_x\psi_x+\chi U\psi_{xx}+\chi V_x\phi_x+\chi V\phi_{xx}+\chi(\phi_x\psi_x)_x,\\
\psi_{xt}=\phi_{xx}.
\end{array}\right.
\end{equation}
Multiplying the first equation of \eqref{eqn-2.15} by $\frac{\phi_{xxx}}{U}$, integrating the result over $[0,\infty)\times[0,t]$ and noting that
\[\begin{split}\frac{\phi_{xt}\phi_{xxx}}{U}&=\left(\frac{\phi_{xt}\phi_{xx}}{U}\right)_x-\frac{\phi_{xxt}\phi_{xx}}{U}-\phi_{xt}\phi_{xx}(\frac{1}{U})_x\\
&=\left(\frac{\phi_{xt}\phi_{xx}}{U}\right)_x-\left(\frac{\phi_{xx}^2}{2U}\right)_t+\frac{sU_x\phi_{xx}^2}{2U^2}
+\frac{U_x\phi_{xt}\phi_{xx}}{U^2},\end{split}\]
we have
\begin{equation}\begin{split}\label{3.36}
&\ii\frac{\phi_{xx}^2}{2U}+\int_0^t\ii \frac{D\phi_{xxx}^2}{U}-\int_0^t\ii \frac{sU_x}{2U^2}\phi_{xx}^2\\
&=-\chi\int_0^t\ii\frac{\phi_{xxx}}{U}(U_x\psi_x+U\psi_{xx}+V_x\phi_x+ V\phi_{xx}+(\phi_x\psi_x)_x)\\&\quad+\int_0^t\ii\frac{U_x\phi_{xt}\phi_{xx}}{U^2}-\int_0^t\frac{\phi_{xt}\phi_{xx}}{U}(0,\tau)d\tau+\ii\frac{\phi_{0xx}^2}{2U}.
\end{split}\end{equation}
By Young's inequality, we get
\[\begin{split}
\chi\int_0^t\ii\left|\phi_{xxx}\psi_{xx}\right|&\leq\frac{D}{4}\int_0^t\ii\frac{\phi_{xxx}^2}{U}
+\frac{\chi^2}{D}\int_0^t\ii U\psi_{xx}^2,\\
\chi\int_0^t\ii\left|\frac{\phi_{xxx}(\phi_x\psi_x)_x}{U}\right|&
\leq\chi\int_0^t\ii\left(\frac{\left|\phi_{xxx}\phi_{xx}\psi_x\right|}{U}
+\frac{\left|\phi_{xxx}\psi_{xx}\phi_x\right|}{U}\right)\\&\leq N(t)\int_0^t\ii\frac{\phi_{xxx}^2}{U}
+CN(t)\int_0^t\ii\left(\frac{\phi_{xx}^2}{U}+\frac{\psi_{xx}^2}{U}\right),\end{split}\]
where we have used $\norm{\psi_x(\cdot,t)}_{L^\infty}\leq N(t)$, $\norm{\phi_x(\cdot,t)}_{L^\infty}\leq N(t)$.
Thus, in view of \eqref{3-6}, the first term on the RHS of \eqref{3.36} satisfies
\begin{equation}\begin{split}\label{3.37}
&-\chi\int_0^t\ii\frac{\phi_{xxx}}{U}(U_x\psi_x+U\psi_{xx}+V_x\phi_x+ V\phi_{xx}+(\phi_x\psi_x)_x)\\&
\leq(\frac{D}{2}+N(t))\int_0^t\ii\frac{\phi_{xxx}^2}{U}+\frac{\chi^2}{D}\int_0^t\ii U\psi_{xx}^2+CN(t)\int_0^t\ii\frac{\psi_{xx}^2}{U}\\&\quad+C\int_0^t\ii (\psi_{x}^2+\phi_{x}^2+\frac{\phi_{xx}^2}{U}).
\end{split}\end{equation}
By $\eqref{eqn-2.15}_1$ and \eqref{3-6} again, the second term on the RHS of \eqref{3.36} can be estimated as
\begin{equation*}\begin{split}
&\int_0^t\ii\frac{U_x\phi_{xt}\phi_{xx}}{U^2}\\&=\int_0^t\ii\frac{U_x}{U^2}\cdot[D\phi_{xxx}+\chi U_x\psi_x+\chi U\psi_{xx}+\chi V_x\phi_x+\chi V\phi_{xx}+\chi(\phi_x\psi_x)_x]\phi_{xx}\\
&\leq\frac{D}{4}\int_0^t\ii\frac{\phi_{xxx}^2}{U}+\frac{\chi^2}{D}\int_0^t\ii U\psi_{xx}^2+CN(t)\int_0^t\ii\frac{\psi_{xx}^2}{U}\\&\quad+C\int_0^t\ii (\frac{\psi_{x}^2}{U}+\phi_{x}^2+\frac{\phi_{xx}^2}{U})
\end{split}\end{equation*}
where we have used the following inequality obtained by the Cauchy-Schwarz inequality:
$$\left|\frac{\chi U_x\phi_{xx}\psi_{xx}}{U}\right|\leq\frac{s\chi\left|\phi_{xx}\psi_{xx}\right|}{D}
\leq\frac{s^2\phi_{xx}^2}{DU}+\frac{\chi^2U\psi_{xx}^2}{D}.$$

We next estimate the boundary term in \eqref{3.36}. In view of $\eqref{nonlinear system}_1$ and \eqref{3.6},
\[D\phi_{xx}=A'(t)-\chi(U\psi_x+V\phi_x+\phi_x\psi_x) \ \text{ at } x=0.\]
Thus, it follows that
\begin{equation}\label{3.39n}
\begin{split}
-\int_0^t\frac{\phi_{xt}\phi_{xx}}{U}(0,\tau)d\tau=&-\frac{1}{D}\int_0^t\frac{A'(\tau)\phi_{xt}}{U}(0,\tau)d\tau
+\frac{\chi}{D}\int_0^t\phi_{xt}\psi_x(0,\tau)d\tau\\&-\frac{s\chi}{2D}\int_0^t(\phi_{x}^2)_t(0,\tau)d\tau+\frac{\chi}{2D}\int_0^t\frac{(\phi_{x}^2)_t\psi_x}{U}(0,\tau)d\tau.
\end{split}
\end{equation}
Next we estimate the terms on the RHS of \eqref{3.39n}. First by the integration by parts, the first term on the RHS of \eqref{3.39n} equals to
\begin{eqnarray*}\begin{aligned}
&-\frac{1}{D}\int_0^t\frac{A'(\tau)\phi_{xt}}{U}(0,\tau)d\tau\\
&=-\frac{A'(t)\phi_{x}}{DU(-st+\a-\b)}+\frac{A'(0)\phi_{0x}}{DU(\a-\b)} +\frac{1}{D}\int_0^t\phi_{x}\left(\frac{A''(\tau)}{U(-s\tau+\a-\b)}+\frac{sA'(\tau)U_x}{U^2(-s\tau+\a-\b)}\right).
\end{aligned}\end{eqnarray*}
With \eqref{At}, $A''(t)=s^2U_x(-st+\a-\b)-\eta'(t)$, and the fact $U(-st+\a-\b)>U(0)=\frac{u_-}{2}$, we have that
\begin{equation*}\begin{split}
-\frac{1}{D}\int_0^t\frac{A'(\tau)\phi_{xt}}{U}(0,\tau)d\tau&\leq CN(t)\left(|\eta(t)-\eta_-|+|\eta(0)-\eta_-|+e^{-\frac{s}{D}\b}\right)\\
&\quad+CN(t)\int_0^t(|\eta'(\tau)|+|\eta(\tau)-\eta_-|+e^{\frac{s}{D}(-s\tau+\a-\b)})d\tau\\&\leq C(\delta+e^{-\frac{s}{D}\b})
\end{split}\end{equation*}
By \eqref{2}, the second term on the RHS of \eqref{3.39n} equals to
\begin{equation*}\begin{split}
\frac{\chi}{D}\int_0^t\phi_{xt}\psi_x(0,\tau)d\tau=\frac{\chi}{D}\phi_{x}\psi_x(0,t)-\frac{\chi}{D}\phi_{x}\psi_x(0,0)-\frac{\chi}{D}\int_0^t\phi_{x}\phi_{xx}(0,\tau)d\tau.
\end{split}\end{equation*}
Notice that for a function $f\in H^1(0,\infty)$, it holds that
\[f^2(0)=-\int_0^\infty(f^2(x))_xdx=-2\int_0^\infty ff_xdx\leq\frac{1}{\epsilon}\int_0^\infty f^2dx+\epsilon\int_0^\infty f_x^2dx.\]
Thus, by Lemma \ref{H^1 estimate}, the following holds:
\begin{equation*}\begin{split}
\frac{\chi}{D}\int_0^t\phi_{xt}\psi_x(0,\tau)d\tau&\leq\frac{\chi}{2D}\left[(\phi_x^2+\psi_x^2)(0,t)+(\phi_x^2+\psi_x^2)(0,0)
+\int_0^t(\phi_x^2(0,\tau)+\phi_{xx}^2(0,\tau))\right]\\
&\leq\epsilon\int_0^\infty [(\phi_{xx}^2+\psi_{xx}^2)+(\phi_{0xx}^2+\psi_{0xx}^2)]dx+\epsilon\int_0^t\int_0^\infty (\phi_{xx}^2+\phi_{xxx}^2)dxd\tau\\
&\quad+\frac{C}{\epsilon}\int_0^\infty [(\phi_{x}^2+\psi_{x}^2)+(\phi_{0x}^2+\psi_{0x}^2)]dx+\frac{C}{\epsilon}\int_0^t\int_0^\infty (\phi_{x}^2+\phi_{xx}^2)dxd\tau\\
&\leq\epsilon\int_0^\infty(\phi_{xx}^2+\psi_{xx}^2)dx+\epsilon\int_0^t\int_0^\infty\phi_{xxx}^2dxd\tau\\
&\quad+C\left(\|\psi_{0x}\|_{w_0}^2+\|\phi_0\|^2_{1,w_0}
+\|\psi_0\|_2^2+e^{-\frac{s\b}{D}}+\delta\right).
\end{split}\end{equation*}
Similarly, the last two terms of \eqref{3.39n} satisfy
\begin{equation*}\begin{split}
-\frac{s\chi}{2D}\int_0^t(\phi_{x}^2)_t(0,\tau)d\tau&=-\frac{s\chi}{2D}\left[\phi_x^2(0,t)-\phi_x^2(0,0)\right]\\
&\leq\epsilon\int_0^\infty \phi_{xx}^2dx+C\left(\|\psi_{0x}\|_{w_0}^2+\|\phi_0\|^2_{1,w_0}+\|\phi_{xx}\|^2
+\|\psi_0\|_1^2+e^{-\frac{s\b}{D}}+\delta\right),
\end{split}\end{equation*}
and by \eqref{2}
\begin{equation*}\begin{split}
\frac{\chi}{2D}\int_0^t\frac{(\phi_{x}^2)_t\psi_x}{U}(0,\tau)d\tau&=\frac{\chi}{2D}\left[\left(\frac{\phi_x^2\psi_x(0,t)}{U}-\frac{\phi_x^2\psi_x(0,0)}{U}\right)
-\int_0^t\left(\frac{\phi_x^2\phi_{xx}}{U}+\phi_x^2\psi_x\left(\frac{1}{U}\right)_t\right)\right]\\
&\leq\epsilon N(t)\int_0^\infty (\phi_{xx}^2+\psi_{xx}^2)dx+\epsilon N(t)\int_0^t\int_0^\infty\phi_{xxx}^2dxd\tau\\&\quad+C\left(\|\psi_{0x}\|_{w_0}^2+\|\phi_0\|^2_{1,w_0}+\|\phi_{xx}\|^2
+\|\psi_0\|_1^2+e^{-\frac{s\b}{D}}+\delta\right).
\end{split}\end{equation*}
Then substituting above results into \eqref{3.39n} yields
\begin{equation}\begin{split}\label{3.40}
-\int_0^t\frac{\phi_{xt}\phi_{xx}}{U}(0,\tau)d\tau&\leq\epsilon(2+N(t))\int_0^\infty (\phi_{xx}^2+\psi_{xx}^2)dx+\epsilon(1+N(t))\int_0^t\int_0^\infty\phi_{xxx}^2dxd\tau\\&\quad+C\left(\|\psi_{0x}\|_{w_0}^2+\|\phi_0\|^2_{1,w_0}
+\|\psi_0\|_2^2+e^{-\frac{s\b}{D}}+\delta\right).
\end{split}\end{equation}
Feeding \eqref{3.37}-\eqref{3.40} into \eqref{3.36}, we obtain
\begin{equation}\begin{split}\label{3.41}
&\ii\left(\frac{1}{2U}-3\epsilon\right)\phi_{xx}^2+\int_0^t\ii \left(\frac{D}{4U}-2\epsilon\right)\phi_{xxx}^2\\
&\leq3\epsilon\int_0^\infty \psi_{xx}^2dx+\frac{2\chi^2}{D}\int_0^t\ii U\psi_{xx}^2+CN(t)\int_0^t\ii\frac{\psi_{xx}^2}{U}\\&\quad+C\left(\|\psi_{0x}\|_{w_0}^2+\|\phi_0\|^2_{2,w_0}+\|\phi_{xx}\|^2
+\|\psi_0\|_1^2+e^{-\frac{s\b}{D}}+\delta\right).\end{split}\end{equation}

We next estimate $\int_0^t\ii U\psi_{xx}^2$. Multiplying  the first equation of \eqref{eqn-2.15} by $\psi_{xx}$, we have
\begin{equation}\label{3.42}
\chi U\psi_{xx}^2+D\phi_{xxx}\psi_{xx}=\phi_{xt}\psi_{xx}-(\chi U_x\psi_x+\chi V_x\phi_x+\chi V\phi_{xx}+\chi(\phi_x\psi_x)_x)\psi_{xx}.
\end{equation}
Since the second equation of \eqref{eqn-2.15} gives
\begin{equation}\label{3.43}
\psi_{xxt}=\phi_{xxx},
\end{equation} it follows that
\[\begin{split}&D\phi_{xxx}\psi_{xx}=D\psi_{xxt}\psi_{xx}=\frac{D}{2}(\psi_{xx}^2)_t,\\
&\phi_{xt}\psi_{xx}=(\phi_{x}\psi_{xx})_t-\phi_{x}\psi_{xxt}=(\phi_{x}\psi_{xx})_t-\phi_{x}\phi_{xxx}.\end{split}\]
Then integrating \eqref{3.42}  over $[0,\infty)\times[0,t]$ and noting that
\begin{equation*}
\begin{split}
|\phi_{x}\phi_{xxx}|\leq \frac{\phi_{x}^2}{4\epsilon}+\epsilon\phi_{xxx}^2, \ \chi(\phi_x\psi_x)_x\psi_{xx}=\chi\phi_x\psi_{xx}^2+\chi\phi_{xx}\psi_x\psi_{xx},
\end{split}\end{equation*}
we have
\begin{equation*}\begin{split}
&\frac{D}{2}\ii\psi_{xx}^2+\chi\int_0^t\ii U\psi_{xx}^2\\
&\leq\frac{D}{2}\ii\psi_{0xx}^2+\frac{D}{4}\ii\psi_{xx}^2+\frac{1}{D}\ii\phi_x^2+\frac{(1+N(t))\chi}{4}\int_0^t\ii U\psi_{xx}^2\\
&\quad+\epsilon\int_0^t\ii\phi_{xxx}^2+C\left(\int_0^t\ii\frac{\psi_x^2}{U}+\int_0^t\ii\frac{\phi_x^2}{U}
+\int_0^t\ii\frac{\phi_{xx}^2}{U}+N(t)\int_0^t\ii\psi_{xx}^2\right)
\end{split}\end{equation*}
where we have used the fact \eqref{Sobolev}. Then it follows from Lemma \ref{H^1 estimate} that
\begin{equation}\label{eqn-2.28}\begin{split}
&\ii\psi_{xx}^2+\int_0^t\ii U\psi_{xx}^2-\epsilon\int_0^t\ii\phi_{xxx}^2\\&
\leq C\left(\|\psi_0\|_2^2+\|\phi_0\|^2_{1,w_0}+\|\psi_{0x}\|_{w_0}^2+e^{-\frac{s\b}{D}}+\delta
+N(t)\int_0^t\ii \psi_{xx}^2\right).
\end{split}\end{equation}
Now multiplying \eqref{eqn-2.28} by $\frac{4\chi^2}{D}$ and adding the resulting inequality to \eqref{3.41}, we get
\begin{equation*}\begin{split}
&\ii\left[\left(\frac{1}{2U}-3\epsilon\right)\phi_{xx}^2+\left(\frac{4\chi^2}{D}-3\epsilon\right)\psi_{xx}^2\right]
+\int_0^t\ii \left[\left(\frac{D}{4U}-2\epsilon-\frac{4\chi^2\epsilon}{D}\right)\phi_{xxx}^2+\frac{2\chi^2}{D} U\psi_{xx}^2\right]\\
&\leq C\left(\|\psi_{0x}\|_{w_0}^2+\|\phi_0\|^2_{2,w_0}+\|\phi_{xx}\|^2
+\|\psi_0\|_1^2+e^{-\frac{s\b}{D}}+\delta+N(t)\int_0^t\ii\frac{\psi_{xx}^2}{U}\right).\end{split}\end{equation*}
Because $1/U(x-st+\a-\b)>1/u_-$ for $x>0$, now choosing $\epsilon\ll1$, we then have
\begin{equation}\begin{split}\label{3.44}
&\ii\frac{\phi_{xx}^2}{U}+\int_0^t\ii \frac{\phi_{xxx}^2}{U}\\
&\leq C\left(\|\psi_{0x}\|_{w_0}^2+\|\phi_0\|^2_{2,w_0}+\|\phi_{xx}\|^2
+\|\psi_0\|_1^2+e^{-\frac{s\b}{D}}+\delta+N(t)\int_0^t\ii\frac{\psi_{xx}^2}{U}\right),\end{split}\end{equation}
and
\begin{equation}\begin{split}\label{3.45}
&\ii\psi_{xx}^2+\int_0^t\ii U\psi_{xx}^2\\&
\leq C\left(\|\psi_0\|_2^2+\|\phi_0\|^2_{2,w_0}+\|\psi_{0x}\|_{w_0}^2+e^{-\frac{s\b}{D}}+\delta
+N(t)\int_0^t\ii \psi_{xx}^2\right).
\end{split}\end{equation}

To finish the proof of $\eqref{eqn-2.31}$, we only need to estimate the term $\int_0^t\ii \frac{\psi_{xx}^2}{U}$ or equivalently $\int_0^t\ii w\psi_{xx}^2$ owing to \eqref{3.15}. Using the same argument as deriving (\ref{eqn-5.21}), we first have from \eqref{3.45} that
\begin{equation}\label{eqn-5.32}\begin{split}
&\int_0^{st-\a+\b}w\psi_{xx}^2+\int_0^t\int_0^{s\tau-\a+\b} w\psi_{xx}^2\\&
\leq C\left(\|\psi_{0xx}\|^2+\|\phi_0\|^2_{2,w_0}+\|\psi_{0x}\|_{w_0}^2+\|\psi_0\|^2_1
+e^{-\frac{s\b}{D}}+\delta+N(t)\int_0^t\ii\frac{\psi_{xx}^2}{U}\right).
\end{split}\end{equation}
Multiplying \eqref{3.43} by $w\psi_{xx}$, we get
\begin{equation}\label{eqn-2.33}
\Big(\frac{w\psi_{xx}^2}{2}\Big)_t-\frac{w_t\psi_{xx}^2}{2}=w\phi_{xxx}\psi_{xx}.
\end{equation}
Integrating \eqref{eqn-2.33} over $(st-\a+\b,+\infty)\times[0,t]$ and using \eqref{wt}, we obtain
\begin{equation}\label{3.49}
\begin{split}
&\frac{1}{2}\ii w\psi_{xx}^2+\frac{s^2}{4D}\int_0^t\int_{s\tau-\a+\b}^{\infty}w\psi_{xx}^2\\
&\leq\frac{1}{2}\int_{-\infty}^{+\infty}w\psi_{0xx}^2+\frac{s^2}{8D}\int_0^t\int_{s\tau-\a+\b}^{\infty}w\psi_{xx}^2
+\frac{2D}{s^2}\int_0^t\int_{s\tau-\a+\b}^{\infty}w\phi_{xxx}^2.
\end{split}
\end{equation}
It then follows from \eqref{3.44}, (\ref{eqn-5.32}), \eqref{3.49} and  \eqref{3.15} that
\begin{equation*}
\begin{split}
&\ii w\psi_{xx}^2+\int_0^t\ii w\psi_{xx}^2\\
&\leq C\left(\|\phi_0\|^2_{2,w_0}+\|\psi_{0x}\|_{1,w_0}^2+\|\psi_0\|^2_2
+e^{-\frac{s\b}{D}}+\delta+ N(t)\int_0^t\ii w\psi_{xx}^2\right).
\end{split}
\end{equation*}
When $N(t)$ is small enough, the above inequality  gives
\begin{equation*}
\ii w\psi_{xx}^2+\int_0^t\ii w\psi_{xx}^2
\leq C\Big(\|\psi_{0x}\|_{1,w_0}^2+\|\phi_0\|^2_{2,w_0}
+\|\psi_0\|_2^2+e^{-\frac{s\b}{D}}+\delta\Big).
\end{equation*}
Substituting the above inequality into \eqref{3.44} gives  the estimate for $\int_0^\infty\frac{\phi_{xx}^2}{U}+\int_0^t\ii\frac{\phi_{xxx}^2}{U}$, and  finish the proof of Lemma \ref{sec} thereof.
\end{proof}

\begin{proof}[Proof of Proposition \ref{a priori estimate}] The desired estimate \eqref{priori} is a direct consequence of  Lemma 3.3, Lemma 3.4 and Lemma 3.5. \end{proof}

\begin{proof}[Proof of Theorem \ref{global existence}]
In fact we only need to prove \eqref{long-time behavior} in Theorem \ref{global existence} since the rest of assertions follows from Proposition \ref{a priori estimate} directly. From the global estimate \eqref{priori}, we have
\begin{equation*}
\norm{\phi_x(\cdot,t),\psi_x(\cdot,t)}_{1,w}\to 0 ~~\text{as}~~t\to+\infty.
\end{equation*}
Hence, for all $x\in\R_+$, it follows that
\begin{equation*}
\begin{split}
\phi_x^2(x,t)=2\int_{x}^\infty\phi_x\phi_{xx}(y,t)dy
             \leq 2\left(\int_0^\infty\phi_x^2dy\right)^{1/2}\left(\int_0^\infty\phi_{xx}^2dy\right)^{1/2}
             \leq\norm{\phi_x(\cdot,t)}_{1,w}^2\to 0
\end{split}
\end{equation*}
as $t\to+\infty$. Applying the same procedure to $\psi_x$ leads to
\begin{equation*}
\sup\limits_{x\in\R_+}|\psi_x(x,t)|\to0 \text{ as } t\to+\infty.
\end{equation*}
Hence \eqref{long-time behavior} is proved. \end{proof}

\subsection{Proof of main results} We are ready to prove our main results stated in section 2. First Theorem  \ref{thm-1} is a direct consequence of Theorem \ref{global existence} and Lemma \ref{lem2}. Hence it remains only to prove Theorem  \ref{mainth2} by passing the results from $v$ to $c$.

\begin{proof}[Proof of Theorem  \ref{mainth2}]
Recalling the transformation \eqref{transformation} and \eqref{ict}, we have
\[(\ln c_0(x)-\ln \mathcal{C}(x-\b))_x=-\mu(v_0(x)-V(x-\b))=-\mu\Psi_{0x},\]
which gives
\[\ln c_0(x)-\ln \mathcal{C}(x-\b)=-\mu\Psi_{0}.\]
Thus, the assumptions in Theorem \ref{mainth2} verify those of Theorem \ref{thm-1}, and as a result  the problem $\eqref{ph}$-$\eqref{new-boundary data}$ has a unique global solution $(u,v)(x,t)$ satisfying \eqref{regularity} and the asymptotic behavior \eqref{asym}.

We next derive the results for $c$ from $v$. By the second equation of \eqref{omn}, we get
\begin{equation}\label{3.53}
c(x,t)=c_0(x)e^{-\mu\int_0^tu(x,\tau)d\tau}.
\end{equation}
Thus $c(x,t)$ exists globally, and by \eqref{regularity} and \eqref{transformation}, it holds that
\begin{equation*}
\frac{c_x}{c}(x,t)-\frac{\mathcal{C}_x}{\mathcal{C}}(x-st+\a-\b) \in C([0,\infty);H_w^1) \cap L^2((0,\infty);H_w^1).
\end{equation*}
Owing to the fact that $u(\infty,t)=0$, it is easy to see from \eqref{3.53} that
\[c(\infty,t)=c_0(\infty)=c_+, \forall \ t>0.\]
By the transformation \eqref{transformation} and \eqref{2-5}, one deduces that
\[(\ln c(x,t)-\ln \mathcal{C}(x-st+\a-\b))_x=-\mu(v(x,t)-V(x-st+\a-\b))=-\mu\psi_x.\]
Hence
$$c(x,t)=\mathcal{C}(x-st+\a-\b)e^{-\mu\psi(x,t)}.$$
By the Taylor expansion, we then have
\[\begin{split}
c(x,t)-\mathcal{C}(x-st+\a-\b)&=\mathcal{C}(x-st+\a-\b)(e^{-\mu\psi(x,t)}-1)\\&
=\mu\mathcal{C}(x-st+\a-\b)\psi(x,t)\sum_{n=1}^{\infty}\frac{(-1)^{n}}{n!}(\mu\psi)^{n-1}(x,t).\end{split}\]
By Theorem \ref{global existence}, $\|\psi\|_{C([0,\infty); H^2)}$ is small, which implies the series $\sum_{n=1}^{\infty}\frac{(-1)^{n}}{n!}(\mu\psi)^{n-1}(x,t)$ is convergent. Hence
\[c(x,t)-\mathcal{C}(x-st+\a-\b) \in C([0,\infty);H^2).\]
It remains to derive the asymptotic behavior of $c$. By Theorem \ref{global existence}, we have  $\|\psi_x(t)\|_{1, w}\to 0$ as $t \to \infty$ and $\|\psi(t)\|$ is bounded for all $t>0$. Then
\begin{eqnarray*}
\begin{aligned}
\psi^2(x,t)=2\int_{x}^\infty \psi \psi_y(y,t)dy
\leq 2 \left(\int_0^\infty \psi^2dy\right)^{1/2} \left(\int_0^\infty\psi_y^2dy\right)^{1/2} \to 0 \ \text{as} \ t \to \infty,
\end{aligned}
\end{eqnarray*}
which implies $\sup\limits_{x\in\R_+}|\psi(x,t)| \to 0$ as $t \to \infty$. Note that $\mathcal{C}(x-st+\a-\b)$ is bounded by $c_+>0$. Then
\begin{eqnarray*}
\begin{aligned}
\sup\limits_{x\in\R_+}|c(x,t)-\mathcal{C}(x-st+\a-\b)|&=\sup\limits_{x\in\R_+}\mathcal{C}(x-st+\a-\b)|e^{-\mu\psi(x,t)}-1|\\&\leq \sup\limits_{x\in\R_+}c_+|e^{-\mu\psi(x,t)}-1|\to 0 \ \mathrm{as}\ t \to \infty.
\end{aligned}
\end{eqnarray*}
This completes the proof of Theorem  \ref{mainth2}.
\end{proof}

\section*{Acknowledgements}
The authors are grateful to the referee's many insightful comments which lead to improvements of this manuscript.  J. Li's work was partially supported by the National Science Foundation of China (No. 11571066). He is also grateful for the hospitality of Hong Kong Polytechnic University where part of this work was done. The research of Z. Wang was supported by
the Hong Kong RGC GRF grant No. PolyU 153032/15P.


\begin{thebibliography}{99}
\bibitem{Adler} J. Adler,  Chemotaxis in bacteria, {\it Science}, 153: 708-716, 1966.
\bibitem{CPZ1} L. Corrias, B. Perthame, and H. Zaag. A chemotaxis model motivated by angiogenesis. {\it C.
R. Math. Acad. Sci. Paris}, 2:141-146, 2003.
\bibitem{CPZ3} L. Corrias, B. Perthame, and H. Zaag. Global solutions of some chemotaxis and angiogenesis
systems in high space dimensions. {\it Milan J. Math}., 72:1-29, 2004.
\bibitem{Deng}S. Deng, Initial-boundary value problem of a parabolic-hyperbolic system arising from tumor angiogenesis, {\it J. Differential Equations}, 265:863-890, 2018.
\bibitem{DL} C. Deng and T. Li, Well-posedness of a 3D parabolic-hyperbolic Keller-Segel system in the Sobolev space framework, {\it J. Differential Equations}, 257: 1311-1332, 2014.
\bibitem{Davis-Marangell} P.N. Davis, P. van Heijster, R. Marangell, Absolute instabilities of travelling wave solutions in a Keller-Segel model. Nonlinearity 30: 4029-4061, 2017.
\bibitem{Fan-zhao} J. Fan and K. Zhao, Blow up criteria for a hyperbolic-parabolic system arising from chemotaxis.
{\it J. Math. Anal. Appl.,} 394: 687-695, 2012.
\bibitem{Gold}  R.E. Goldstein, Traveling-wave chemotaxis, {\it Phys. Rev. Lett}.,77: 775-778, 1996.
\bibitem{GXZZ}
J. Guo, J.X. Xiao, H.J. Zhao, and C.J. Zhu, Global solutions to a
hyperbolic-parabolic coupled system with large initial data, {\it
Acta Math. Sci. Ser. B Engl. Ed.,} 29: 629-641, 2009.
\bibitem{Hao}   C. Hao,  Global well-posedness for a multidimensional chemotaxis model in critical Besov spaces, {\it Z. Angew Math. Phys.}, 63: 825-834, 2012.


\bibitem{Horst2} D. Horstmann, From 1970 until present: the Keller-Segel model in chemotaxis and its consequences. II. {\it Jahresber. Deutsch. Math.-Verein}. 106:51-69, 2004.
\bibitem{Horst-Stev} D. Horstmann and A. Stevens, A constructive approach to traveling waves in chemotaxis, {\it J. Nonlin. Sci}., 14:1-25, 2004.


\bibitem{Hou2} Q.Q. Hou, C.J. Liu, Y.G. Wang and Z.A. Wang, Stability of boundary layers for a viscous hyperbolic system arising from chemotaxis: one dimensional case, {\it SIAM J. Math. Anal}., 50:3058-3091, 2018.

\bibitem{jin13}
H.Y. Jin, J.Y. Li, and Z.A. Wang, Asymptotic stability of traveling
waves of a chemotaxis model with singular sensitivity, {\it J. Differential Equations,} 255: 193-219, 2013.


\bibitem{Kalinin} Y.V. Kalinin, L. Jiang, Y. Tu, and M. Wu, Logarithmic sensing in {E}scherichia coli bacterial chemotaxis,  {\it Biophysical J.}, 96: 2439-2448, 2009.
\bibitem{KS} E. F. Keller and L. A. Segel, Traveling bands of chemotactic bacteria: A theoretical analysis, {\it J. Theor. Biol.}, 26: 235-248, 1971.
\bibitem{Levine97}
 H.A. Levine and B.D. Sleeman, A system of reaction diffusion equtions arising in the theory of
reinforced random walks, {\it SIAM J. Appl. Math.,} 57: 683-730, 1997.
\bibitem{LSN} H.A. Levine, B.D. Sleeman, and M. Nilsen-Hamilton,  A mathematical model for the roles of pericytes and macrophages in the initiation of angiogenesis. I. the role of protease inhibitors in preventing angiogenesis,  {\it Math. Biosci.}, 168: 71-115, 2000.
\bibitem{Li111}
 D. Li, T. Li, and  K. Zhao, On a hyperbolic-parabolic system modeling chemotaxis, {\it Math. Models Methods Appl. Sci.,} 21:  1631-1650, 2011.
 \bibitem{Li-pan-zhao}
D. Li, R. Pan and K. Zhao, Quantitative decay of a hybrid type chemotaxis model with
large data, {\it Nonlinearity,} 28: 2181-2210, 2015.
\bibitem{Li-Zhao-JDE} H. Li and K. Zhao, Initial-boundary value problems for a system of hyperbolic balance laws arising from chemotaxis, {\it J. Differential Equations}, 258: 302-308, 2015.


\bibitem{Li14}
J.Y. Li, T. Li, and Z.A. Wang, Stability of traveling waves of the
Keller-Segel system with logarithmic sensitivity, {\it Math. Models
Methods Appl. Sci.,} 24 (2014), 2819-2849.
\bibitem{Lij13}
J.Y. Li, L.N. Wang, and K.J. Zhang, Asymptotic stability of a
composite wave of two traveling waves to a hyperbolic-parabolic
system modeling chemotaxis, {\it Math. Methods Appl. Sci.,}
 36: 1862-1877, 2013.
\bibitem{Li112}
T. Li, R.H. Pan, and K. Zhao, Global dynamics of a chemotaxis model on
bounded domains with large data, {\it SIAM J. Appl. Math.,}
72: 417-443, 2012.
\bibitem{Li09}
T. Li and Z.A. Wang, Nonlinear stability of traveling waves to a
hyperbolic-parabolic system modeling chemotaxis, {\it SIAM J. Appl.
Math.,} 70: 1522-1541, 2009.
\bibitem{Li10}
T. Li and Z.A. Wang, Nonlinear stability of large amplitude viscous shock waves of a
generalized hyperbolic-parabolic system arising in chemotaxis,
{\it Math. Models Methods Appl. Sci.,} 20: 1967-1998, 2010.
\bibitem{Li11}
T. Li and  Z.A. Wang, Asymptotic nonlinear stability of traveling waves to conservation
laws arising from chemotaxis,  {\it J. Differential Equations,}
250(2011), 1310-1333.
\bibitem{Li-Wang-MBS}
T. Li and  Z.A. Wang, Steadily propagating waves of a chemotaxis model,
{\it Math. Biosci.,} 240(2012), 161-168.

\bibitem{Lui-Wang} R. Lui and Z.A. Wang, Traveling wave solutions from microscopic to macroscopic chemotaxis models,
{\it J. Math. Biol.}, 61: 739-761, 2010.

\bibitem{MWZ-Indiana-2018}
V. Martinez, Z.A. Wang and K. Zhao,
Asymptotic and viscous stability of large-amplitude solutions of a hyperbolic system arising from biology,
{\it Indiana Univ. Math. J.}, 67:1383-1424, 2018.

\bibitem{MM99}
A. Matsumura and  M. Mei, Convergence to travelling fronts of solutions
of the $p$-system with viscosity in the presence of a boundary, {\it
Arch. Ration. Mech. Anal.,} 146: 1-22, 1999.

\bibitem{Mei-peng-wang}
M. Mei, H. Peng, Z.A. Wang, Asymptotic profile of a parabolic-hyperbolic system with boundary effect arising from tumor angiogenesis,
{\it  J. Differential Equations,} 259: 5168-5191, 2015.

\bibitem{Nadin} J. Nadin, B. Perthame and L. Ryzhik, Traveling waves for the Keller-Segel system with Fisher birth term, {\it Interface Free Bound.}, 10: 517-538, 2008.

\bibitem{Nishida78} T. Nishida, Nonlinear Hyperbolic Equations and Related Topics in Fluid Dynamics, Publ. Math. d'Orsay, vol. 78-02, D\'{e}partement de Math\'{e}matique, Universit\'{e} de Paris-Sud, Orsay, France, 1978.

\bibitem{Ou} C. Ou and W. Yuan, Traveling wavefronts in a volume-filling chemotaxis model, {\it SIAM J. Appl. Dyn. Syst}., 8: 390-416, 2009.
\bibitem{Othmer97}
H.G. Othmer and  A. Stevens, Aggregation, blowup, and collapse: the
ABC's of taxis in reinforced random walks, {\it SIAM J. Appl.
Math.,} 57: 1044-1081, 1997.

\bibitem{PWZ} H. Peng, H. Wen and C.J. Zhu, Global well-posedness and zero diffusion limit of classical solutions to 3D conservation laws arising in chemotaxis, {\it Z. Angew Math. Phys}., 65(2014), 1167-1188.
\bibitem{smoller} J. Smoller, {\it Shock Waves and Reaction-Diffusion Equations}, Spring-Verlag, Berlin, 1994.

\bibitem{Salako-Shen1} R.B. Salako and W. Shen, Spreading speeds and traveling waves of a parabolic-elliptic chemotaxis system with logistic source on $\R^N$,  {\it Discrete Contin. Dyn. Syst.}, 37: 6189-6225, 2017.

\bibitem{Salako-Shen2} R.B. Salako and W. Shen, Existence of traveling wave solutions of parabolic-parabolic chemotaxis systems. {\it Nonlinear Anal. Real World Appl}., 42: 93-119, 2018.


\bibitem{TWW} Y.S. Tao, L.H. Wang, and Z.A. Wang, Large-time behavior of a parabolic-parabolic chemotaxis model with logarithmic sensitivity in one dimension,  {\it Discrete Contin. Dyn. System-Series B.}, 18: 821-845, 2013.

\bibitem{wang12}
Z.A. Wang, Mathematics of traveling waves in chemotaxis: a review
paper, {\it Discrete Contin. Dyn. Syst. Ser. B,} 18: 601-641, 2013.
\bibitem{Wang08}
Z.A. Wang  and T. Hillen, Shock formation in a chemotaxis Model, {\it Math. Methods Appl. Sci.,} 31: 45-70, 2008.
\bibitem{Wang-xiang-yu}
Z.A. Wang, Z. Xiang and P. Yu,
Asymptotic dynamics on a singular chemotaxis system modeling onset of tumor angiogenesis,
{\it J. Differential Equations,} 260: 2225-2258, 2016.

\bibitem{Welch} R. Welch and D. Kaiser, Cell behavior in traveling wave patterns of myxobacteria, {\it Proceedings
of the National Academy of Sciences}, 98:14907-14912, 2001.

\bibitem{Win1} M. Winkler, The two-dimensional Keller-Segel system with singular sensitivity and signal absorption:  Global large-data solutions and their relaxation propertie, {\it Math. Models Methods Appl. Sci}., 26:987-1024, 2016.

\bibitem{Win2} M. Winkler, Renormalized radial large-data solutions to the higher-dimensional Keller-Segel system with singular sensitivity and signal absorption. {\it J. Differential Equations}, 264: 2310-2350, 2018.

\bibitem{zhang07}
M. Zhang and C.J. Zhu, Global existence of solutions to a
hyperbolic-parabolic system, {\it Proc. Amer. Math. Soc.,}  135: 1017-1027, 2007.
\bibitem{zhang-tan-sun}
Y. Zhang, Z. Tan, and M.B. Sun, Global existence and asymptotic behavior of smooth
solutions to a coupled hyperbolic-parabolic system, {\it Nonlinear Analysis: Real World Applications,}
14: 465-482, 2013.

\end{thebibliography}
\end{document}